\documentclass{birkau}

\usepackage{amsmath,amssymb,url}
\usepackage{enumitem}
\usepackage{amsfonts, amssymb, amsthm, amsmath, calc, cancel, cite,color, eucal, fullpage, graphics, graphicx, latexsym, mathdots, multirow, pgfplots, theoremref, tikz,tikz-cd, url}
\usepackage{accents}

\numberwithin{equation}{section}

\theoremstyle{plain}
\newtheorem{theorem}{Theorem}[section]
\newtheorem{lemma}[theorem]{Lemma}
\newtheorem{proposition}[theorem]{Proposition}
\newtheorem{corollary}[theorem]{Corollary}

\theoremstyle{definition}
\newtheorem{definition}[theorem]{Definition}

\newtheorem{remark}[theorem]{Remark}
\newtheorem{example}[theorem]{Example}

\DeclareMathOperator{\id}{Id}
\renewcommand{\phi}{\varphi}
\renewcommand{\epsilon}{\varepsilon}
\renewcommand{\hom}{\text{Hom}}
\DeclareMathOperator{\con}{Con}
\newcommand{\mb}[1]{\mathbf{#1}}
\renewcommand{\vec}[1]{\mathbf{#1}}
\newcommand{\bb}[1]{\mathbb{#1}}
\newcommand{\mrm}[1]{\mathrm{#1}}
\newcommand{\mc}[1]{\mathcal{#1}}
\newcommand{\mt}[1]{\mathtt{#1}}
\newcommand{\A}{\mb{A}}

\newcommand{\B}{\mb{B}}
\newcommand{\C}{\mb{C}}
\newcommand{\D}{\mb{D}}

\newcommand{\F}{\mb{F}}

\newcommand{\I}{\mb{I}}

\DeclareMathOperator{\K}{\mc{K}^{\circ}}
\renewcommand{\O}{\mc{O}}
\newcommand{\N}{\mb{N}}
\newcommand{\n}{\mb{n}}
\newcommand{\U}{\mb{U}}
\renewcommand{\P}{\mb{P}}
\newcommand{\Q}{\mb{Q}}
\newcommand{\X}{\mrm{X}}
\newcommand{\x}{\vec{x}}

\renewcommand{\u}{\vec{u}}
\newcommand{\w}{\vec{w}}

\newcommand{\0}{\mb{0}}
\newcommand{\1}{\mb{1}}
\newcommand{\bdot}{\boldsymbol{\cdot}}
\DeclareMathOperator{\ch}{Ch}
\renewcommand{\emptyset}{\varnothing}
\newcommand{\meet}{\wedge}
\newcommand{\join}{\vee}

\makeatletter
\providecommand*{\cupdot}{%
  \mathbin{%
    \mathpalette\@cupdot{}%
  }%
}
\newcommand*{\@cupdot}[2]{%
  \ooalign{%
    $\m@th#1\sqcup$\cr
    \sbox0{$#1\sqcup$}%
    \dimen@=\ht0 %
    \sbox0{$\m@th#1\cdot$}%
    \advance\dimen@ by -\ht0 %
    \dimen@=.5\dimen@
    \hidewidth\raise\dimen@\box0\hidewidth
  }%
}

\providecommand*{\bigcupdot}{%
  \mathop{%
    \vphantom{\bigsqcup}%
    \mathpalette\@bigcupdot{}%
  }%
}
\newcommand*{\@bigcupdot}[2]{%
  \ooalign{%
    $\m@th#1\bigsqcup$\cr
    \sbox0{$#1\bigsqcup$}%
    \dimen@=\ht0 %
    \advance\dimen@ by -\dp0 %
    \sbox0{\scalebox{2}{$\m@th#1\cdot$}}%
    \advance\dimen@ by -\ht0 %
    \dimen@=.5\dimen@
    \hidewidth\raise\dimen@\box0\hidewidth
  }%
}
\makeatother

\begin{document}

\title[Spectral properties of cBCK-algebras]{Spectral properties of cBCK-algebras}

\author[C. M. Evans]{C. Matthew Evans}
\address{Mathematics Department\\
Oberlin College\\Oberlin, OH 44074\\USA}
\urladdr{https://sites.google.com/view/mattevans}
\email{mevans4@oberlin.edu}



\begin{abstract}
In this paper we study prime spectra of commutative BCK-algebras. We give a new construction for commutative BCK-algebras using rooted trees, and determine both the ideal lattice and prime ideal lattice of such algebras. We prove that the spectrum of any commutative BCK-algebra is a locally compact generalized spectral space which is compact if and only if the algebra is finitely generated as an ideal. Further, we show that if a commutative BCK-algebra is involutory, then its spectrum is a Priestley space. Finally, we consider the functorial properties of the spectrum and define a functor from the category of commutative BCK-algebras to the category of distributive lattices with zero. We give a partial answer to the question: what distributive lattices lie in the image of this functor?
\end{abstract}

\maketitle

\section{Introduction}\label{sec:intro}

The class of BCK-algebras was introduced in 1966 by Imai and Is\'{e}ki \cite{II66} as the algebraic semantics for a non-classical logic having only implication. This implicational calculus is evidently due to Tarski and Bernays, but Is\'{e}ki also credits Meredith in \cite{iseki66}. The origin of the terms B, C, and K is the combinatory logic of Sch\"{o}nfinkel \cite{schon24} and Curry \cite{curry30} from the 1920's and 1930's.

The class $\mt{BCK}$ of BCK-algebras is not a variety (\cite{wronski83}), but many subclasses do form varieties. In this paper we focus on the variety of commutative BCK-algebras, denoted $\mt{cBCK}$, which has ties to many other algebraic structures including MV-algebras, lattice-ordered Abelian groups, BCI-algebras, AF $C^\ast$-algebras, \L ukasiewicz algebras, commutative integral residuated lattices, and others. We note, for exmaple, that the variety of bounded commutative BCK-algebras is term-equivalent to the variety of MV-algebras \cite{mundici86}.

The core of this paper deals with a topological representation for commutative BCK-algebras. The idea of representing an algebraic structure with a topological space dates back to Stone's pioneering work \cite{stone36} which provides a dual equivalence between the category of Boolean algebras, $\mt{BA}$, and the category of Stones spaces, $\mt{Stone}$. This was later extended in \cite{pries70} to a dual equivalence between $\mt{BDL}$, the category of bounded distributive lattices, and the category $\mt{Pries}$ of Priestley spaces; as is well-known these two equivalences both arise as natural dualities (\cite{clarkdavey98}).

One may wonder whether this type of natural duality is possible for $\mt{cBCK}$. The representation we develop here will not lead to a dual equivalence: many non-isomorphic algebras will have the same spectrum (up to homeomorphism). Further, in \cite{niederkorn00}, Niederkorn showed that the variety $\mt{bcBCK}$ of bounded commutative BCK-algebras is not dualisable. The signature for the variety $\mt{bcBCK}$ is not the same as that of $\mt{cBCK}$, so this does not rule out the possibility that $\mt{cBCK}$ could be dualisable. While we would conjecture that $\mt{cBCK}$ is not dualisable, to the author's knowledge this is an open problem. However, the representing spaces are still interesting objects in their own right, and point toward some interesting connections between commutative BCK-algebras and commutative rings.

The organization of this paper is as follows: in section 2 we give an overview of the necessary background of commutative BCK-algebras and the basics of their ideal theory. We also provide some examples that will be used throughout the paper.

In section 3 we describe two constructions for building commutative BCK-algebras, and characterize both their ideal lattices and prime ideal lattices. The first construction is a disjoint union-type construction that seems to have first appeared in \cite{it76}. Despite this being a known construction, the characterization of the prime ideal lattices is new. The second construction involves the use of rooted trees to define commutative BCK-algebras, and this construction is new. We also show in this section that any finite subdirectly irreducible distributive p-algebra occurs as the ideal lattice of a commutative BCK-algebra.

In section 4 we define the spectrum of a commutative BCK-algebra and consider its topological properties. In particular, we show that the spectrum of a commutative BCK-algebra is a locally compact generalized spectral space which is compact if and only if the algebra is finitely generated as an ideal. We also show that when the algebra is involutory, the spectrum is a Priestley space. 

In section 5 we discuss the functoriality of the spectrum and define a functor $\K\X$ from $\mt{cBCK}$ to $\mt{DL_0}$, the category of distributive lattices with 0. By focusing our attention on Noetherian spectra, we give a partial answer to the question: what lattices lie in the image of $\K\X$? In giving this partial answer, we are finding distributive lattices that occur both as the lattices of compact open subsets of the spectra of some cBCK-algebras, and as the ideal lattices of some cBCK-algebras.

We note that this paper is an adaptation of the author's dissertation \cite{evans20}, and that the results of this paper are contained in \cite{evans20} in some form, with the exception of section \ref{disjoint union in gspec}.

\section{Preliminaries}

\begin{definition} A \textit{commutative BCK-algebra} is an algebra $\langle A; \bdot, 0\rangle$ of type $(2,0)$ such that 
\begin{enumerate}
\item[]\hspace{-1cm} (cBCK1)\; $(x\bdot y)\bdot z=(x\bdot z)\bdot y$
\item[]\hspace{-1cm} (cBCK2)\; $x\bdot(x\bdot y)=y\bdot(y\bdot x)$
\item[]\hspace{-1cm} (cBCK3)\; $x\bdot x=0$
\item[]\hspace{-1cm} (cBCK4)\; $x\bdot 0=x$
\end{enumerate}
for all $x,y,z\in A$.
\end{definition}
Throughout, we will write $\A=\langle A; \bdot, 0\rangle$, and we will refer to commutative BCK-algebras as cBCK-algebras. Denote the variety of cBCK-algebras by $\mt{cBCK}$. If $\A=\langle A; \bdot_A, 0_A\rangle$ and $\B=\langle B; \bdot_B, 0_B\rangle$ are cBCK-algebras, we say a function $h\colon\A\to \B$ is a \textit{BCK-homomorph\-ism} if $h(x\bdot_A y)=h(x)\bdot_B h(y)$ for all $x,y\in A$. We note that any BCK-homomorphism is also 0-preserving due to (cBCK3). The notation $\mt{cBCK}$ will also denote the category with cBCK-algebras as objects and BCK-homomorphisms as morphisms.

For the elementary properties of cBCK-algebras, we point the reader to Is\'{e}ki and Tanaka's introductory papers \cite{it78} and \cite{it76}, Tana\-ka's paper \cite{tanaka75}, Romanowska and Traczyk's paper \cite{rt80}, Traczyk's paper \cite{traczyk79}, Yutani's paper \cite{yutani77}, and the text \cite{mj94} by Meng and Jun.

We collect here a few important properties.

\begin{proposition}[\cite{it78}]\label{basic properties} Let $\A$ be a cBCK-algebra.
\begin{enumerate}
\item $\A$ is partially ordered via: $x\leq y$ if and only if $x\bdot y=0$.
\item $0\bdot x=0$ for all $x\in A$, so $0$ is the least element in $\A$.
\item The operation $\bdot$ is right isotone; that is, if $x\leq y$, then $z\bdot x\leq z\bdot y$.
\item The operation $\bdot$ is left antitone; that is, if $x\leq y$, then $y\bdot z\leq x\bdot z$.
\item The term operation $x\meet y:=y\bdot (y\bdot x)$ is the greatest lower bound of $x$ and $y$.
\item $x\bdot y\leq x$ with equality if and only if $x\meet y=0$.
\item $\A$ is a semilattice with respect to $\meet$.
\end{enumerate}
\end{proposition} The identity (cBCK2) tells us $x\meet y=y\meet x$; these algebras are called ``commutative'' because of this.

We say a cBCK-algebra $\A$ is \textit{bounded} if there is an element $1\in A$ such that $x\bdot 1=0$ for all $x\in A$, so $x\leq 1$ for all $x\in A$. The class of bounded cBCK-algebras may be considered as a variety as well; that is, an algebra $\A=\langle A; \bdot, 0, 1\rangle$ of type $(2,0,0)$ is a bounded commutative BCK-algebra if it satisfies (cBCK1)-(cBCK4) as well as $x\bdot 1=0$ for all $x\in A$. This variety will be denoted $\mt{bcBCK}$.

Given a bounded cBCK-algebra $\A$ the term operation \[x\join y := 1\bdot\bigl( (1\bdot x)\meet (1\bdot y)\bigr)\,,\] gives the least upper bound of $x$ and $y$. Is\'{e}ki and Tanaka showed in \cite{it78} that the term-reduct $\A^{\text{d}}=\langle A; \meet, \join\rangle$ is a lattice, while Traczyk showed in \cite{traczyk79} that $\A^\text{d}$ is a distributive lattice. 


\subsection{Examples}

There are many natural examples of cBCK-algebras. We focus our attention on those examples which will be most useful for our purposes.

The set of non-negative reals $\bb{R}_{\geq 0}$ becomes a cBCK-algebra via the operation $x\bdot y=\max\{x-y, 0\}$. We denote this algebra by $\mb{R}^+$. This truncated difference is the prototypical operation for a generic cBCK-algebra.

From this algebra we obtain important subalgebras. The non-negative integers $\bb{N}_0:=\bb{N}\cup \{0\}$ is a cBCK-subalgebra of $\mb{R}^+$ which we will denote $\N_0$.

If we let $I=[0,1]$ denote the unit interval in $\bb{R}$, then we obtain a cBCK-subalgebra $\I$ of $\mb{R}^+$. Putting $Q=I\cap \bb{Q}$, we have another cBCK-subalgebra of $\mb{R}^+$ which we will denote $\Q$.

For $k\in\bb{N}$, let $C_k=\bigl\{0, \frac{1}{k}, \frac{2}{k},\ldots, \frac{k-1}{k}, 1\bigr\}$. This is a cBCK-subalgebra of $\I$ we will denote $\C_k$. In particular, $\C_1$ is just the two-element cBCK-algebra with universe $\{0,1\}$.

\begin{remark} The variety $\mt{bcBCK}$ is generated by $\I$; that is, $\mt{bcBCK}=\text{HSP}(\I)$. This is essentially the content of Chang's Completeness Theorem for many-valued logic, see \cite{chang58} and \cite{chang59}. Chang's proof is in the language of MV-algebras, but Mundici showed in \cite{mundici86} that MV-algebras and bcBCK-algebras are term-equivalent. Further, we also have that $\mt{bcBCK}=\text{HSP}(\C_1, \C_2, \C_3, \ldots)$ (see Proposition 8.1.2 of \cite{CDM00}), but $\mt{bcBCK}$ is not finitely generated (see \cite{cornish80}).
\end{remark}


\subsection{Ideals}

\begin{definition} A subset $I\subseteq A$ of a cBCK-algebra $\A$ is an \textit{ideal} if $0\in I$ and the following implication is satisfied: if $x\bdot y\in I$ and $y\in I$, then $x\in I$.
\end{definition}

Every ideal is a down-set: take $y\in I$ and $x\in \A$ with $x\leq y$. Then $x\bdot y=0\in I$, and since $y\in I$ we must have $x\in I$. 

\begin{definition} Let $\A$ be a cBCK-algebra and $P$ an ideal of $\A$.
\begin{enumerate}
\item We say $P$ is \textit{proper} if $P\neq \A$.
\item We say $P$ is a \textit{prime ideal} if it is proper and $x\meet y\in P$ implies $x\in P$ or $y\in P$.
\item We say $P$ is an \textit{irreducible ideal} if, whenever $I\cap J=P$ for ideals $I$ and $J$, we have $I=P$ or $J=P$.
\item We say $P$ is a \textit{meet-prime ideal} if, whenever $I\cap J\subseteq P$ for ideals $I$ and $J$, we have $I\subseteq P$ or $J\subseteq P$.
\end{enumerate}
\end{definition}

Pa\l asinski proved in \cite{palasinski81} that prime ideals, irreducible ideals, and meet-prime ideals all coincide in a cBCK-algebra.

Given a cBCK-algebra $\A$, we will denote the collection of all ideals of $\A$ by $\id(\A)$ and the collection of all prime ideals of $\A$ by $\X(\A)$. As lattices, $\id(\A)\cong\con(\A)$. A proof can be found in \cite{rt80}, \cite{yutani77}, or \cite{AT77}, but we give the idea here: for $I\in\id(\A)$, define $\theta_I\subseteq A\times A$ by $(x,y)\in \theta_I$ if and only if $x\bdot y\in I$ and $y\bdot x\in I$. The map sending $I\mapsto \theta_I$ is a lattice isomorphism. The inverse is $\theta\mapsto [0]_\theta$, where $[0]_\theta$ is the equivalence class of $0$. 

The lattice $\id(\A$) is also known to be distributive, see Lemma 3.2 of \cite{rt80}. From this and the previous paragraph, it follows that $\mt{cBCK}$ is a congruence-distributive variety. 

Given a subset $S$ of a cBCK-algebra $\A$, the \textit{ideal generated by $S$}, denoted $(S]$, is the smallest ideal containing $S$. Is\'{e}ki and Tanaka provide a very nice characterization of $(S]$.

\begin{theorem}[\cite{it76}, Theorem 3]\label{iseki ideal theorem} Let $S$ be a subset of a cBCK-algebra $\A$. Then $x\in(S]$ if and only if there exist $s_1,\ldots, s_n\in S$ such that 
\[\bigl(\cdots \bigl((x\bdot s_1)\bdot s_2\bigr)\bdot \cdots \bdot s_{n-1}\bigr)\bdot s_n=0\,.\]
\end{theorem}

If $S=\{x_1, x_2,\ldots, x_k\}$, we may write $(x_1, x_2, \ldots, x_k]$ rather than $(S]$. In particular, for singleton subsets $\{x\}$ we will write $(x]$.

For any cBCK-algebra $\A$, the subsets $\{0_A\}$ and $A$ are always ideals. If these are the only ideals, we say $\A$ is \textit{simple}.

For $x,y\in\A$, we define the notation $x\bdot y^n$ for $n\in \bb{N}_0$ recursively by
\begin{align*}
x\bdot y^0&=x\\
x\bdot y^n&=\bigl(x\bdot y^{n-1}\bigr)\bdot y\,.
\end{align*} Since $x\bdot y\leq x$, any pair $x,y\in\A$ gives us a decreasing sequence
\[x\bdot y^0\geq x\bdot y^1\geq x\bdot y^2\geq \cdots \geq x\bdot y^n\geq\cdots\,.\]

If the underlying poset of a cBCK-algebra is totally ordered, we will call it a \textit{cBCK-chain}. For example, all of the algebras $\mb{R}^{+}$, $\N_0$, $\I$, $\Q$, and $\C_k$ for $k\in \bb{N}$ are cBCK-chains. 

\begin{proposition} A cBCK-chain is simple if and only if, for any $x,y\in \A$, $y\neq 0$, there is a natural number $n$ such that $x\bdot y^n=0$.
\end{proposition}

This result is stated without proof in \cite{rt82}. For the sake of completeness we provide a proof here.

\begin{proof} Assume first that $\A$ is simple, and take $x,y\in \A$ with $y\neq 0$. Consider the ideal $(y]$. By simplicity we must have $(y]=\A$ since $y\neq 0$. But this means $x\in(y]$; therefore, by Theorem \ref{iseki ideal theorem}, there exists $n\in\bb{N}$ such that $x\bdot y^n=0$. 

On the other hand, assume for any pair $x,y\in\A$ with $y\neq 0$ that there exists $n\in\bb{N}$ such that $x\bdot y^n=0$. Let $I$ be a non-zero ideal of $\A$. Take $z\in \A$ and $y\neq 0$ in $I$. By hypothesis there is some $k\in\bb{N}$ such that $z\bdot y^k=0\in I$. Since $y\in I$, we repeatedly apply the ideal property to obtain $z\in I$. Hence, $I=\A$ and $\A$ is simple. 
\end{proof}

From this it follows that $\mb{R}^+$, $\mb{N}_0$, $\I$, $\Q$, and $\C_k$ are all simple. For some cBCK-chains that are not simple, see Examples \ref{chain of length n} and \ref{countable chain}.

We note that the ideals of any cBCK-chain are linearly ordered and the prime ideals of any cBCK-chain are precisely the proper ideals; see Lemmas 2.3.1 and 2.3.2 of \cite{evans20} for proofs.


\subsection{Involutory algebras}

Let $\A=\langle A; \bdot, 0\rangle$ be a cBCK-algebra. 

\begin{definition} For $S\subseteq A$, the \textit{annihilator of $S$} is \[S^\ast:=\{a\in \A\mid a\wedge s=0\text{ for all } s\in S\}\,.\] 
\end{definition} Aslam and Thaheem prove in \cite{AT91} that $(-)^\ast$ is a Galois connection and that $S^\ast$ is an ideal of $\A$.

\begin{definition} We say that an ideal $I$ of $\A$ is \textit{involutory} if $I=I^{\ast\ast}$. We say the algebra $\A$ is \textit{involutory} if every ideal is involutory.
\end{definition}

The zero ideal $\{0\}$ and $\A$ itself are always involutory, and therefore any simple cBCK-algebra is an involutory algebra. We say that $\A$ \textit{satisfies the descending chain condition} if the sequence 
\[x\geq x\bdot y\geq x\bdot y^2\geq \cdots \geq x\bdot y^n\geq \cdots \,.\]
stabilizes for any pair $x,y\in \A$; that is, for each pair $x,y\in \A$ there is some $n\in\bb{N}_0$  such that $x\bdot y^n=x\bdot y^{n+1}$. One can show that $\A$ is involutory if and only if it satisfies the descending chain condition; see Theorem 3.10 of\cite{AT91} together with Theorem 3.3 of \cite{xin01}.

For an integer $n\geq 1$, consider the identity 
\begin{align*}
x\bdot y^n = x\bdot y^{n+1}\tag{$\text{E}_n$}\,.
\end{align*} Any cBCK-algebra $\A$ satisfying ($\text{E}_1$) also must satisfy the identity $x\bdot (y\bdot x)=x$ and is said to be \textit{implicative}. We note also that any bounded implicative BCK-algebra is a Boolean algebra (\cite{it78}). Varieties of BCK-algebras satisying ($\text{E}_n$) are discussed in detail in \cite{cornish80} and \cite{dyrda87}.


\begin{lemma} If a cBCK-algebra $\A$ is finite, locally finite, or satisfies ($\text{E}_n$) for some $n$, then $\A$ is involutory.
\end{lemma}

\begin{proof} If $\A$ is finite or satisfies ($\text{E}_n$) for some $n$, clearly any decreasing sequence will stabilize.

Suppose $\A$ is locally finite. Take $x,y\in\A$ and consider $\langle x,y\rangle$, the subalgebra generated by $x$ and $y$. This subalgebra contains the sequence $(x\bdot y^n)_{n\in\bb{N}_0}$, but it is a finitely-generated subalgebra and hence finite. Thus, the sequence $(x\bdot y^n)_{n\in\bb{N}_0}$ must stabilize.

In all three cases, the algebra satisfies the descending chain condition and is therefore involutory.
\end{proof}


\section{Two constructions}

In this section we describe two methods of building cBCK-algebras. For each construction we characterize the ideals and prime ideals.

\subsection{cBCK-unions}\label{unions}

Let $\Lambda$ be an index set and $\bigl\{\A_\lambda\bigr\}_{\lambda\in\Lambda}$ a family of cBCK-algebras. Suppose further that $A_\lambda\cap A_\mu=\{0\}$ for $\lambda\neq \mu$ and let $U$ be the union of the $A_\lambda$'s. We will use the notation $U=\bigcupdot_{\lambda\in\Lambda} A_\lambda$.

Equipping $U$ with the operation
\[x\bdot y=\begin{cases} x\boldsymbol{\cdot}_\lambda y & \text{ if $x,y\in A_\lambda$}\\x & \text{ otherwise}\end{cases}\;,\] where $\bdot_\lambda$ is the BCK-operation in $\A_\lambda$, yields a new cBCK-algebra which we will denote as $\U=\bigcupdot_{\lambda\in\Lambda} \A_\lambda$. We will refer to $\U$ as a \textit{cBCK-union}.

That this construction does indeed yield a cBCK-algebra is proven in \cite{it76} in the case $|\Lambda|=2$. Extending the proof to arbitrary $\Lambda$ is tedious but straightforward; a full proof can be found in the author's dissertation \cite{evans20} (Proposition 2.2.1). Note that if $x\in\A_\lambda$ and $y\in\A_\mu$ with $\lambda\neq \mu$, then $x\meet y=0$ in $\U$.

Ideals and prime ideals in a cBCK-union are very well behaved.

\begin{proposition}[\cite{yutani80}, Propositions 3 and 4]\label{ideals_in_union} A subset $I\subseteq U$ is an ideal of $\U=\bigcupdot_{\lambda\in\Lambda} \A_\lambda$ if and only if $I=\bigcupdot_{\lambda\in\Lambda} I_\lambda$, where $I_\lambda\in\id(\A_\lambda)$. For a given ideal $I$, this decomposition is unique.\end{proposition}

\begin{theorem}\label{primes_in_union} Let $\U=\bigcupdot_{\lambda\in\Lambda}\A_\lambda$. An ideal $P$ of $\U$ is prime if and only if there exists $\mu\in\Lambda$ and $Q\in\mrm{X}(\A_\mu)$ so that \[P=\bigcupdot_{\lambda\in\Lambda}\A_{\lambda,\mu}^Q\,,\] where $\A_{\lambda,\mu}^Q=\left.\begin{cases}\A_\lambda&\text{if $\lambda\neq\mu$}\\Q&\text{if $\lambda=\mu$}\end{cases}\right\}\,.$
\end{theorem}

\begin{proof} First suppose $P=\bigcupdot_{\lambda\in\Lambda}\A_{\lambda,\mu}^Q$ for some $\mu\in\Lambda$ and $Q\in\mrm{X}(\A_\mu)$. By Proposition \ref{ideals_in_union}, we see that $P$ is an ideal of $\U$. Suppose $x\meet y\in P$ but $x,y\notin P$. Then $x,y\in\A_\mu\setminus Q$, but since $Q$ is prime in $\A_\mu$ we must have $x\meet y\notin Q$. This is a contradiction since $x\meet y\notin Q$ implies $x\meet y\notin P$. So $P$ must be prime. 

On the other hand, let $P$ be a prime ideal of $\U$. Then $P=\bigcupdot_{\lambda\in\Lambda} I_\lambda$ for ideals $I_\lambda\in\id(\A_\lambda)$. If there are indices $\alpha\neq\beta$ such that $I_\alpha\neq \A_\alpha$ and $I_\beta\neq \A_\beta$, choose $x\in \A_\alpha\setminus I_\alpha$ and $y\in \A_\beta\setminus I_\beta$. Then $x\meet y=0\in P$, but $x,y\notin P$, a contradiction. So $I_\lambda=\A_\lambda$ for all but at most one index. But prime ideals are proper, so we have $I_\lambda=\A_\lambda$ for all but exactly one index, say $\mu$. We claim that $I_\mu$ is a prime ideal of $\A_\mu$. 

Take $a,b\in\A_\mu\setminus I_\mu$. Then $a,b\notin P$. Since $P$ is prime we have $a\meet b\notin P$. Thus, $a\meet b\notin I_\mu$, meaning $I_\mu$ is prime in $\A_\mu$. Therefore, $P$ is of the desired form.
\end{proof}

We take a moment to consider when a cBCK-union is involutory. In Theorem \ref{invol implies pries} we will see that an algebra $\A$ being involutory gives a great deal of information about its spectrum.

\begin{lemma}\label{ann_of_union} If $I=\bigcupdot_{\lambda\in\Lambda} I_\lambda\in\id(\U)$, then \[I^\ast=\bigl(\,\bigcupdot_{\lambda\in\Lambda} I_\lambda\,\bigr)^\ast=\bigcupdot_{\lambda\in\Lambda} I_\lambda^\ast\,.\]
\end{lemma}

\begin{proof} Take $x\in I^\ast$. Since $x\in\U$, we have $x\in \A_\alpha$ for some $\alpha\in\Lambda$. Then for any $y\in I_\alpha$, we have $x\meet y=0$, so $x\in I_\alpha^\ast\subseteq \bigcupdot_{\lambda\in\Lambda} I_\lambda^\ast$. Thus, $I^\ast\subseteq \bigcupdot_{\lambda\in\Lambda} I_\lambda^\ast$.

For the other inclusion, take $x\in\bigcupdot_{\lambda\in\Lambda} I_\lambda^\ast$. Then $x\in I_\alpha^\ast$ for some $\alpha\in\Lambda$, and $x\meet y=0$ for all $y\in I_\alpha$. If we take $z\in I_\beta$ for any $\beta\neq\alpha$, then $x\meet z=0$ since $x\in \A_\alpha$ and $z\in \A_\beta$. Hence $x\in\bigl(\,\bigcupdot_{\lambda\in\Lambda} I_\lambda\,\bigr)^\ast=I^\ast$, and thus $\bigcupdot_{\lambda\in\Lambda} I_\lambda^\ast\subseteq I^\ast$.
\end{proof}

\begin{theorem}\label{involutory_union} The algebra $\U=\bigcupdot_{\lambda\in\Lambda} \A_\lambda$ is involutory if and only if each $\A_\lambda$ is involutory.
\end{theorem}

\begin{proof} This follows from Proposition \ref{ideals_in_union} and Lemma \ref{ann_of_union}.
\end{proof}


\subsection{cBCK-algebras associated to trees}

Let $T$ be a rooted tree; we will use Greek letters to indicate elements of the vertex set $V(T)$, and in particular we will use $\lambda$ to indicate the root of $T$. Denote by $\bb{Z}^T$ the set of all functions $V(T)\to\bb{Z}$. Let $A^T$ be the subset of $\bb{Z}^T$ consisting of all functions $\u\colon V(T)\to\bb{Z}$ with finitely many non-zero entries and where the first non-zero entry along every root-based path is positive.

For an element $\u\in A^T$ and a vertex $\alpha\in V(T)$, we will write $u_\alpha$ to indicate the value of $\u$ at $\alpha$. For a root-based path $p$ we will write $\u_p$ for the ``sub-tuple'' of $\u$ corresponding to the values of $\u$ along the path $p$. If $p$ is an interval in $T$, say $p=[\lambda, \alpha]$, we may write $\u_{[\lambda, \alpha]}$ rather than $\u_p$. On occassion we will use other standard interval notations, particularly $[\lambda, \alpha)$, indicating the interval from $\lambda$ to $\alpha$, but excluding $\alpha$. We will write $\0$ for the zero function.  For the sake of clarity, we provide a small example. 

\begin{example} Figure \ref{fig:tree1} shows a rooted tree $T$ and an element $\u\in A^T$.
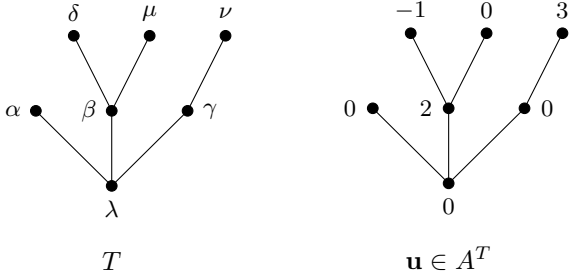
\begin{figure}[h]
\centering
\begin{tikzpicture}

\filldraw (0,0) circle (2pt);
\filldraw (0,1) circle (2pt);
\filldraw (-1,1) circle (2pt);
\filldraw (1,1) circle (2pt);
\filldraw (-.5,2) circle (2pt);
\filldraw (.5,2) circle (2pt);
\filldraw (1.5,2) circle (2pt);

\draw [-] (0,1) -- (0,0);
\draw [-] (-1,1) -- (0,0); 
\draw [-] (1,1) -- (0,0);
\draw [-] (0,1) -- (-.5,2);
\draw [-] (0,1) -- (.5,2);
\draw [-] (1,1) -- (1.5,2);

	\node at (0,-.3) {\small $\lambda$};
	\node at (-1.3, 1) {\small $\alpha$};
	\node at (-.3, 1) {\small $\beta$};
	\node at (1.3, 1) {\small $\gamma$};
	\node at (-.5, 2.3) {\small $\delta$};
	\node at (.5, 2.3) {\small $\mu$};
	\node at (1.5, 2.3) {\small $\nu$};	
	
\node at (0,-1) {$T$};	
\end{tikzpicture}
\hspace{1cm}
\begin{tikzpicture}

\filldraw (0,0) circle (2pt);
\filldraw (0,1) circle (2pt);
\filldraw (-1,1) circle (2pt);
\filldraw (1,1) circle (2pt);
\filldraw (-.5,2) circle (2pt);
\filldraw (.5,2) circle (2pt);
\filldraw (1.5,2) circle (2pt);

\draw [-] (0,1) -- (0,0);
\draw [-] (-1,1) -- (0,0); 
\draw [-] (1,1) -- (0,0);
\draw [-] (0,1) -- (-.5,2);
\draw [-] (0,1) -- (.5,2);
\draw [-] (1,1) -- (1.5,2);

	\node at (0,-.3) {\small $0$};
	\node at (-1.3, 1) {\small $0$};
	\node at (-.3, 1) {\small $2$};
	\node at (1.3, 1) {\small $0$};
	\node at (-.5, 2.3) {\small $-1$};
	\node at (.5, 2.3) {\small $0$};
	\node at (1.5, 2.3) {\small $3$};	
	
\node at (0,-1) {$\u\in A^T$};	
\end{tikzpicture}
\caption{A rooted tree $T$ and an element $\u\in A^T$}\label{fig:tree1}
\end{figure}
So $u_\beta=2$ and $u_\delta=-1$ while $\u_{[\lambda,\beta]}=(0,2)$ and $\u_{[\lambda,\delta]}=(0,2,-1)$.
\end{example}

For $\u\in A^T$ and any root-based path $p$, we see that $\u_p$ is a $\bb{Z}$-valued $|p|$-tuple. Let $\leq_\ell$ denote the lexicographic order on the set of $|p|$-tuples.

We will also have occasion to use the fact that $V(T)$ is partially ordered as well. We will write $\alpha\leq_T\beta$ to indicate that the vertices $\alpha$ and $\beta$ are comparable with $\alpha$ at or below $\beta$ in $T$. That is, $\alpha\leq_T\beta$ if and only if $\alpha$ is an ancestor of $\beta$. For example, $\lambda\leq_T\alpha$ for all $\alpha\in V(T)$.

Define an operation on $A^T$ as follows: for $\u, \vec{v}\in A^T$ and $\alpha\in V(T)$, 
\[(\u\bdot \vec{v})_\alpha=
\begin{cases}u_\alpha-v_\alpha & \text{ if $\u_{[\lambda, \alpha]}>_\ell \vec{v}_{[\lambda, \alpha]}$}\\
0 & \text{ if $\u_{[\lambda, \alpha]}\leq_\ell \vec{v}_{[\lambda, \alpha]}$}
\end{cases}\,.\] Under this operation, $A^T$ becomes a cBCK-algebra which we will denote $\A^T$. The proof of this is straightforward, though rather tedious. We refer the reader to the author's dissertation \cite{evans20} (Proposition 2.4.2) for the proof. The cBCK-order that arises can be described as follows: $\u\leq \vec{v}$ if and only if $\u_p\leq_\ell \vec{v}_p$ for every root-based path $p$. See again \cite{evans20}.

\begin{remark} In \cite{cornish81}, Cornish also used rooted trees to construct cBCK-algebras: if a rooted tree $T$ (thought of as a lower semilattice with 0) admits a valuation $v$ in a cBCK-algebra $\C$, then one can define a BCK-operation on $T$ yielding a commutative BCK-algebra. In some cases, Cornish's construction gives back known examples; for instance, his Example 1.3 is the same as the cBCK-union defined in subsection \ref{unions} of the present paper. This example first appears in \cite{it76}.

However, Cornish's construction is quite different from the construction presented here. Given a rooted tree $T$ with a valuation $v\colon T\to \C$ in a cBCK-algebra $\C$, Cornish's construction produces a cBCK-structure on the \textit{labels} of the vertices of $T$. In particular, if we begin with a finite rooted tree, the corresponding cBCK-algebra is finite. By contrast, using the construction defined above, each element of $\A^T$ is itself a $\bb{Z}$-labeling of $T$, and $\A^T$ is always infinite.

To give a concrete example of the disinction between these two constructions, consider the two-element chain $\ch_2=\{\lambda, \alpha\}$ with $\lambda<\alpha$. Cornish's construction applied to $\ch_2$ gives (an isomorphic copy of) the unique two-element cBCK-algebra, which is simple (and Boolean). Using our construction defined above, the algebra $\A^{\ch_2}$ is countably infinite and has a proper, non-trivial ideal. See Example \ref{chain of length n} below.
\end{remark}

For the next lemma, recall the term operation \[\u\meet\vec{v}=\vec{v}\bdot(\vec{v}\bdot\u)=\u\bdot(\u\bdot\vec{v})\,.\]

\begin{lemma}\label{meet in A^T} For any $\u,\vec{v}\in\A^T$ and any root-based path $p$, we have 
\[(\u\meet\vec{v})_p=\left.\begin{cases}\u_p & \text{ if $\u_p\leq_\ell\vec{v}_p$}\\\vec{v}_p & \text{ if $\u_p>_\ell\vec{v}_p$}\end{cases}\right\}\,.\] Consequently $(\u\meet\vec{v})_p=\u_p\meet_\ell\vec{v}_p$, where $\meet_\ell$ is the meet with respect to the lexicographic order on the set of $|p|$-tuples.
\end{lemma}

\begin{proof} Assume first that $\u_p\leq_\ell \vec{v}_p$. Then $(\u\bdot \vec{v})_\alpha=0$ for all $\alpha\in p$ and so $(\u\bdot \vec{v})_p=\0_p$. From this we see that $(\u\meet \vec{v})_\alpha=\bigl(\u\bdot(\u\bdot \vec{v})\bigr)_\alpha=u_\alpha$ for all $\alpha\in p$, and hence $(\u\meet\vec{v})_p=\u_p$.

Next, assume instead that $\u_p>_\ell\vec{v}_p$. Then there is some vertex $\beta\in p$ such that $\vec{v}_{[\lambda, \beta)}=\u_{[\lambda,\beta)}$ and $v_\beta<u_\beta$. So for vertices $\gamma\in[\lambda,\beta)$ we have $(\u\bdot \vec{v})_\gamma=0$, meaning $\bigl(\u\bdot(\u\bdot\vec{v})\bigr)_\gamma=u_\gamma=v_\gamma$. But then for vertices $\delta\in p\setminus [\lambda,\beta)$ we have $(\u\bdot \vec{v})_\delta=u_\delta-v_\delta$ and so $\bigl(\u\bdot(\u\bdot\vec{v})\bigr)_\delta=u_\delta-(u_\delta-v_\delta)=v_\delta$. Thus, we have $\bigl(\u\bdot(\u\bdot\vec{v})\bigr)_\alpha=v_\alpha$ for all $\alpha\in p$, and $\bigl(\u\bdot(\u\bdot\vec{v})\bigr)_p=\vec{v}_p$ as desired.
\end{proof}

The next several results describe the ideals of $\A^T$ and their general behavior. Let $\bb{P}(T)$ denote the set of all root-based paths in $T$. Consider the binary relation $\zeta\subseteq \A^T\times \bb{P}(T)$ given by
\[\zeta =\{\,(\u, p)\in \A^T\times\bb{P}(T)\mid \u_p=\0_p\,\}\,.\] This relation induces a Galois connection:
\begin{align*}
&\text{for $U\subseteq \A^T$, put } \mc{P}(U)=\{p\in \bb{P}(T)\mid \u_p=\0_p \text{ for all $\u\in U$}\}\\
&\text{for $R\subseteq \bb{P}(T)$, put } I(R)=\{\u\in \A^T\mid \u_p=\0_p \text{ for all $p\in R$}\}\,.
\end{align*} Notice that $I(\emptyset)=\A^T$ and, though it is an abuse of notation, $I(T):=I\bigl(\bb{P}(T)\bigr)=\{\0\}$. If $R$ is a singleton set, say $R=\{p\}$, we will simply write $I(p)$, and if $p=[\lambda,\alpha]$ we will write $I(\alpha)$. 

\begin{proposition}\label{I_P is an ideal} For any collection of root-based paths $R$, the set $I(R)$ is an ideal of $\A^T$.
\end{proposition}

\begin{proof} Clearly $\0\in I(R)$. Suppose $\u\bdot\vec{v}\in I(R)$ and $\vec{v}\in I(R)$, and pick $p\in R$. Then $\vec{v}_p=\0_p$, and we have $\vec{v}_p\leq_\ell\u_p$. This gives $(\u\bdot\vec{v})_\alpha=u_\alpha-v_\alpha=u_\alpha$ for each $\alpha\in p$, and hence $(\u\bdot\vec{v})_p=\u_p$. But $(\u\bdot\vec{v})_p=\0_p$ since $\u\bdot\vec{v}\in I(R)$, and therefore $\u_p=\0_p$. Since $p$ was arbitrary, $\u_p=\0_p$ for all $p\in R$, and $\u\in I(R)$.
\end{proof}

\begin{theorem}\label{ideals in A^T} For every ideal $J$ of $\A^T$, we have $J=I(\mc{P}(J))$. In particular, every ideal of $\A^T$ has the form $I(R)$ for some collection $R$ of root-based paths.
\end{theorem}

\begin{proof} Let $J$ be an ideal of $\A^T$. We claim that $J=I\bigl(\mc{P}(J)\bigr)$. The inclusion $\subseteq$ follows from the fact that $\mc{P}(-)$ and $I(-)$ form a Galois connection.

For the other inclusion, take $\u\in I\bigl(\mc{P}(J)\bigr)$ with $\u\neq\0$. Let $\alpha\in V(T)$ be such that $\u_{[\lambda,\alpha)}=\0_{[\lambda,\alpha)}$ but $u_\alpha\neq 0$; then $u_\alpha>0$. Set $p:=[\lambda,\alpha]$. Then $p\notin\mc{P}\bigl(I\bigl(\mc{P}(J)\bigr)\bigr)=\mc{P}(J)$. So there is $\vec{v}\in J$ such that $\vec{v}_p\neq \0_p$. Let $\beta\in[\lambda,\alpha]$ be such that $\vec{v}_{[\lambda,\beta)}=\0_{[\lambda,\beta)}$ and $v_\beta\neq 0$. Note that $v_\beta>0$.

Let $k=\bigl\lceil\frac{u_\beta}{v_\beta}\bigr\rceil$ and put $n=k+1$. We claim that $(\u\bdot\vec{v}^n)_q=\0_q$ for any root-based path $q$ having $p$ as a prefix.

If $\beta <_T\alpha$, then $u_\beta=0$ and $n=1$. But $v_\beta>0$ tells us $\u_p<_\ell \vec{v}_p$, meaning $\u_q<_\ell \vec{v}_q$ for any root-based $q$ having $p$ as a prefix, and so $(\u\bdot\vec{v}^n)_q=(\u\bdot\vec{v})_q=\0_q$ for any such path $q$.

So suppose $\beta=\alpha$. We know that $k$ is the smallest positive integer such that $\frac{u_\alpha}{v_\alpha}<k$, or equivalently $u_\alpha-k\,v_\alpha<0$. By the definition of $\bdot$, this means $(\u\bdot\vec{v}^k)_\alpha=0$. Since $v_\alpha=v_\beta>0$, we see that $(\u\bdot\vec{v}^k)_q<_\ell \vec{v}_q$ for any root-based $q$ containing $p$ as a prefix, and therefore $(\u\bdot\vec{v}^n)_q=\bigl((\u\bdot\vec{v}^k)\bdot\vec{v}\bigr)_q=\0_q$ for any such path $q$. This proves the claim.

By definition of $\A^T$, the element $\u$ has finitely many non-zero vertices, and so in particular there are finitely many vertices $\alpha$ such that $u_\alpha\neq 0$ but $\u_{[\lambda, \alpha)}=\0_{[\lambda,\alpha)}$. Said differently, there are only finitely many paths along which $\u$ takes on a non-zero value. Enumerate these vertices $\alpha_1, \alpha_2, \ldots, \alpha_m$. By the argument in the preceding paragraphs, for each $\alpha_i$ we can find an element $\vec{v}_i\in J$ and positive integer $l_i$ such that $(\u\bdot\vec{v}_i^{l_i})_q=\0_q$ for any root-based path $q$ containing $[\lambda, \alpha_i]$ as a prefix. But then
\[\bigl(\cdots\bigl((\u\bdot\vec{v}_1^{l_1})\bdot\vec{v}_2^{l_2}\bigr)\bdot \cdots \bigr)\bdot\vec{v}_m^{l_m}=\0\in J\,.\] Since each $\vec{v}_i\in J$, repeatedly applying the ideal property gives us $\u\in J$ as well. 

Hence, $I\bigl(\mc{P}(J)\bigr)\subseteq J$, and therefore $I\bigl(\mc{P}(J)\bigr)= J$.
\end{proof}

For two root-based paths $p$ and $q$ we will write $p\subseteq q$ to indicate that $p$ is a prefix of $q$.

\begin{proposition}\label{behavior of ideals in A^T} Let $R, R_1, R_2\subseteq \bb{P}(T)$ and $p, q\in\bb{P}(T)$.
\begin{enumerate}
\item If $p\in R$, then $I(R)\subseteq I(p)$.

\item $I(R)=\bigcap_{p\in R} I(p)$.

\item If $p\subseteq q$, then $I(q)\subseteq I(p)$.

\item $I(p)\join I(q)=I(p\cap q)$, and consequently $I(\alpha)\join I(\beta)=I(\alpha\meet_T\beta)$ for vertices $\alpha,\beta\in V(T)$.

\item $I(R_1)\join I(R_2)=\bigcap_{p\in R_1}\bigcap_{q\in R_2} I(p\cap q)$.

\item $I(R_1)\cap I(R_2)=I(R_1\cup R_2)$.

\end{enumerate}
\end{proposition}

\begin{proof}\hfill\break
\noindent(1), (2), and (3) are clear.

\noindent(4) From (3), we know $I(p), I(q)\subseteq I(p\cap q)$, and thus $I(p)\join I(q)\subseteq I(p\cap q)$. For the other inclusion, take $\u\in I(p\cap q)$ so that $\u_{p\cap q}=\0_{p\cap q}$. Define $\vec{v}\in\A^T$ by \[v_\alpha=\left.\begin{cases}0 &\text{ if $\alpha\in p$}\\ u_\alpha&\text{ if $\alpha\in T\setminus p$}\end{cases}\right\}\,.\] Note that
\begin{align*}
(\u\bdot\vec{v})_\alpha &= 0 \text{ for all $\alpha\in T\setminus(p\cap q^c)$}\\
(\u\bdot\vec{v})_\alpha &= u_\alpha \text{ for all $\alpha\in p\cap q^c$}\,,
\end{align*} and that $\vec{v}\in I(p)$. Now define $\w\in\A^T$ by
\[w_\alpha=\left.\begin{cases}0 &\text{ if $\alpha\in q$}\\ (\u\bdot\vec{v})_\alpha&\text{ if $\alpha\in T\setminus q$}\end{cases}\right\}\] and notice that
\begin{align*}
\bigl((\u\bdot\vec{v})\bdot\w\bigr)_\alpha &= 0 \text{ for all $\alpha\in T\setminus(p\cap q^c)$}\\
\bigl((\u\bdot\vec{v})\bdot\w)_\alpha &= 0 \text{ for all $\alpha\in p\cap q^c$ since $p\cap q^c\subseteq T\setminus q$}\,,
\end{align*} and $\w\in I(q)$. Hence $(\u\bdot\vec{v})\bdot\w=\0$ with $\w,\vec{v}\in I(p)\cup I(q)$, so by Theorem \ref{iseki ideal theorem} we have $\u\in\bigl(I(p)\cup I(q)\bigr]=I(p)\join I(q)$. Thus $I(p\cap q)\subseteq I(p)\join I(q)$, and therefore $I(p\cap q)= I(p)\join I(q)$.

\noindent(5) Using (2) and (4), together with the fact that $\id(\A^T)$ is distributive, we have
\begin{align*}
I(R_1)\join I(R_2)
=\Bigl(\bigcap_{p\in R_1} I(p)\Bigr)\join\Bigl(\bigcap_{q\in R_2} I(q)\Bigr)
&=\bigcap_{p\in R_1}\bigcap_{q\in R_2} I(p)\join I(q)\\
&=\bigcap_{p\in R_1}\bigcap_{q\in R_2} I(p\cap q)\,.
\end{align*}

\noindent(6) Since $R_1,R_2\subseteq R_1\cup R_2$, an easy extension of (1) above gives $I(R_1\cup R_2)\subseteq I(R_1), I(R_2)$. Hence $I(R_1\cup R_2)\subseteq I(R_1)\cap I(R_2)$. For the other inclusion, take $\u\in I(R_1)\cap I(R_2)$ so that $\u_p=\0_p$ for all $p\in R_1$ and $\u_q=\0_q$ for all $q\in R_2$. Then $\u_p=\0_p$ for all $p\in R_1\cup R_2$, so $\u\in I(R_1\cup R_2)$ and the result follows.
\end{proof}

\begin{theorem}\label{prime ideals in A^T} An ideal of $\A^T$ is prime if and only if it can be realized as $I(p)$ for a root-based path $p$.
\end{theorem}

\begin{proof} We first prove $I(p)$ is a prime ideal. By Proposition \ref{I_P is an ideal} we know $I(p)$ is an ideal, and we note that it is a proper ideal since $u_\lambda=0$ for all $\u\in I(p)$. Suppose $\u\meet\vec{v}\in I(p)$. Then $(\u\meet\vec{v})_p=\0_p$ and by Lemma \ref{meet in A^T} we have either $\u_p=\0_p$ or $\vec{v}_p=\0_p$. That is, either $\u\in I(p)$ or $\vec{v}\in I(p)$. 

Assume now that $I$ is a prime ideal. By Theorem \ref{ideals in A^T}, we must have $I=I(R)$ for some collection $R$ of root-based paths. We note $R$ must be non-empty, for otherwise $I=\A^T$, which is a contradiction since prime ideals are proper. And if $R$ is a singleton set, we're done.

So suppose $|R|\geq 2$, and assume that $R$ contains two root-based paths, say $p$ and $q$, neither of which is a prefix of the other. Then we may choose vertices $\alpha\in p\setminus (p\cap q)$ and $\beta\in q\setminus (p\cap q)$. Define $\u\in \A^T$ to be zero everywhere except $u_\alpha=1$, and similarly define $\vec{v}\in \A^T$ to be zero everywhere except $v_\beta=1$. Then certainly $\u\meet\vec{v}=\0\in I$, but $\u\notin I$ and $\vec{v}\notin I$. Thus, $I$ is not prime, a contradiction. 

Therefore, the root-based paths appearing in $R$ must form an ascending chain $p_1\subseteq p_2\subseteq p_3\subseteq \cdots$. If this chain is finite, stopping at some $p_n\in R$, then we have $p_i\subseteq p_n$ for all $i$, meaning $I(p_n)\subseteq I(p_i)$ for all $i$, and consequently $I(R)=I(p_n)$ by Proposition \ref{behavior of ideals in A^T}(2). If this chain is infinite, let $p$ represent the infinite-length root-based path carved out by the $p_i$'s. Note that $\u_p=\0_p$ if and only if $\u_{p_i}=\0_{p_i}$ for all $i$, and so again $I(R)=I(p)$.
\end{proof}

\begin{proposition}\label{X(A^T)=P(T)} As posets, $\X(\A^T)\cong \bb{P}(T)^\partial$, where $\bb{P}(T)^\partial$ is the order-dual of $\bb{P}(T)$.
\end{proposition}

\begin{proof} Define a map $\phi\colon\bb{P}(T)\to \X(\A^T)$ by $\phi(p)=I(p)$. If $p$ and $q$ are two root-based paths with $p\subseteq q$, then $I(p)\supseteq I(q)$ by Proposition \ref{behavior of ideals in A^T}(3), and so $\phi(p)\supseteq \phi(q)$.

On the other hand, suppose $\phi(q)\subseteq \phi(p)$. If $p\not\subseteq q$ then there is a vertex $\alpha$ along $p$ which is not on the path $q$. Define $\u\in\A^T$ to be zero everywhere except $u_\alpha=1$. Then $\u_q=\0_q$ and so $\u\in I(q)$, but $\u\notin I(p)$, meaning $I(q)\not\subseteq I(p)$, a contradiction. Thus, we must have $p\subseteq q$.

The argument above can be modified slightly to show that $\phi$ is injective: if $p\neq q$, then $I(p)\neq I(q)$. Finally, this map is surjective by Theorem \ref{prime ideals in A^T}. Hence, $\phi$ is an order-anti-isomorphism.
\end{proof}

\begin{corollary}\label{X(A^T) = T^d} If $T$ is a finite rooted tree, then $\X(\A^T)\cong T^\partial$ as posets. 
\end{corollary}

\begin{proof} Suppose $T$ is finite. Then any root-based path in $T$ is finite and hence determined by its terminal vertex. Define $\psi\colon T\to \bb{P}(T)$ by $\psi(\alpha)=[\lambda, \alpha]$. That this is a bijection is straightforward, and certainly $\alpha\leq_T \beta$ if and only if $[\lambda, \alpha]\subseteq [\lambda, \beta]$, meaning $\psi$ is an order-isomorphism. But then $T^\partial \cong \bb{P}(T)^\partial\cong \X(\A^T)$ by Proposition \ref{X(A^T)=P(T)} above.
\end{proof}


\subsection{Examples}

\begin{example}\label{chain of length n} Let $\ch_n$ denote the chain of length $n-1$ viewed as a rooted tree. So $\ch_n$ has $n$ vertices. The algebra $\A^{\ch_n}$ is a cBCK-chain, so the ideals are linearly ordered. That is, the ideal lattice $\id(\A^{\ch_n})$ is itself a chain and the prime ideals are exactly the proper ideals. From Theorem \ref{ideals in A^T} we see that $\A^{\ch_n}$ has $n+1$ ideals, and thus it has $n$ prime ideals. That is, the chain $\X(\A^{\ch_n})$ is isomorphic to the $n$-element chain, $\n$, which we could also see immediately from Corollary \ref{X(A^T) = T^d}.
\end{example}

\begin{example}\label{countable chain} Let $\ch_\infty$ denote a rooted tree that is a countably infinite chain. As in the previous example, the algebra $\A^{\ch_\infty}$ is a cBCK-chain and therefore $\id(\A^{\ch_\infty})$ is itself a chain and $\X(\A^{\ch_\infty})=\id(\A^{\ch_\infty})\setminus \{\A^{\ch_\infty}\}$. For a root-based path $p$ in $\ch_\infty$, let $\ell(p)$ denote the length of $p$. Let $\bb{N}_0^\infty=\bb{N}_0\cup\{\infty\}$, where $k<\infty$ for all $k\in\bb{N}_0$, and note that $\ell\colon \bb{P}(\ch_\infty)\to \bb{N}_0^\infty$ is an order-isomorphism. Hence, by Proposition \ref{X(A^T)=P(T)} we have $\X(\A^{\ch_\infty})\cong (\bb{N}_0^\infty)^\partial$ as posets. This algebra has the peculiar property that $\id(\A^{\ch_\infty})\cong (\bb{N}_0^\infty)^\partial$ as well.

\end{example}

In the following examples, we adopt the notation $I_R$ in place of $I(R)$ for $R\subseteq \bb{P}(T)$. In particular, for an interval $[\lambda, \alpha]$, the ideal $I(\alpha)$ will be denoted $I_\alpha$.

\begin{example}\label{T_2} Figure \ref{fig:tree2} shows a tree we will call $T_2$ and the Hasse diagram for the ideals of $\A^{T_2}$, with $\X(\A^{T_2})$ indicated in red, obtained by applying Theorem \ref{ideals in A^T} and Proposition \ref{prime ideals in A^T}.
\begin{figure}[h]
\centering
\begin{tikzpicture}

\filldraw (0,0) circle (2pt);
\filldraw (-.75,1) circle (2pt);
\filldraw (.75,1) circle (2pt);

\draw [-] (-.75,1) -- (0,0);
\draw [-] (.75,1) -- (0,0); 

	\node at (0,-.3) {\small $\lambda$};
	\node at (-.75, 1.3) {\small $\alpha_1$};
	\node at (.75, 1.3) {\small $\alpha_2$};	
		
\end{tikzpicture}
\hspace{1cm}
\begin{tikzpicture}

\filldraw (0,0) circle (2pt);
\filldraw[red] (-.75,1) circle (2pt);
\filldraw[red] (.75,1) circle (2pt);
\filldraw[red] (0,2) circle (2pt);
\filldraw (0,3) circle (2pt);

\draw [-] (-.75,1) -- (0,0);
\draw [-] (.75,1) -- (0,0); 
\draw [-,red] (-.75,1) -- (0,2);
\draw [-,red] (.75,1) -- (0,2); 
\draw [-] (0,2) -- (0,3);

	\node at (0,-.4) {\small $\{\0\}$};
	\node at (-1.2, 1) {\small $I_{\alpha_1}$};
	\node at (1.2, 1) {\small $I_{\alpha_2}$};
	\node at (.35,2.1) {\small $I_\lambda$};
	\node at (0, 3.3) {\small $\A^{T_2}$};	
	
\end{tikzpicture}
\caption{The tree $T_2$ and $\id(\A^{T_2})$}\label{fig:tree2}
\end{figure}
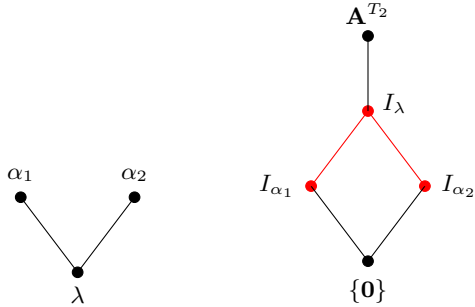 

Similarly, consider $T_3$ as in Figure \ref{fig:tree3}. For notational brevity, let $I_j=I_{\alpha_j}$ and $I_{jk}=I_{\alpha_j}\cap I_{\alpha_k}$. Similar computations give the Hasse diagram for $\id(\A^{T_3})$, where again  $\X(\A^{T_3})$ is indicated in red.
\begin{figure}[h]
\centering
\begin{tikzpicture}

\filldraw (0,0) circle (2pt);
\filldraw (0,1) circle (2pt);
\filldraw (-1,1) circle (2pt);
\filldraw (1,1) circle (2pt);

\draw [-] (-1,1) -- (0,0);
\draw [-] (1,1) -- (0,0); 
\draw [-] (0,1) -- (0,0);

	\node at (0,-.3) {\small $\lambda$};
	\node at (-1, 1.3) {\small $\alpha_1$};
	\node at (0,1.3) {\small $\alpha_2$};
	\node at (1, 1.3) {\small $\alpha_3$};	
	
\end{tikzpicture}
\hspace{1cm}
\begin{tikzpicture}[scale=1.2]

\node[circle,fill=black,inner sep=0pt,minimum size=4pt] (min) at (0,0) {};
\node[circle,fill=black,inner sep=0pt,minimum size=4pt] (b) at (0,1) {};
\node[circle,fill=red,inner sep=0pt,minimum size=4pt] (e) at (0,2) {};
\node[circle,fill=red,inner sep=0pt,minimum size=4pt] (g) at (0,3) {};
\node[circle,fill=black,inner sep=0pt,minimum size=4pt] (max) at (0,4) {};
\node[circle,fill=black,inner sep=0pt,minimum size=4pt] (a) at (-1,1) {};
\node[circle,fill=black,inner sep=0pt,minimum size=4pt] (c) at (1,1) {};
\node[circle,fill=red,inner sep=0pt,minimum size=4pt] (d) at (-1,2) {};
\node[circle,fill=red,inner sep=0pt,minimum size=4pt] (f) at (1,2) {};
  \draw (d) -- (a) -- (min) -- (b);
  \draw (f) -- (c) -- (min);
  \draw (a) -- (e) -- (c);
  \draw[preaction={draw=white, -,line width=6pt}] (d) -- (b) -- (f);
  \draw (g) -- (max);
  \draw[-,red] (d)--(g) -- (f);
  \draw[-,red] (g)--(e);
  
	\node at (0,-.4) {\small $\{\0\}$};
	\node at (-1.2, .8) {\small $I_{12}$};
	\node at (1.3, .8) {\small $I_{23}$};
	\node at (.35,.8) {\small $I_{13}$};
	\node at (-1.25, 2) {\small $I_1$};
	\node at (1.25, 2) {\small $I_3$};
	\node at (.35, 2) {\small $I_2$};
	\node at (.35,3) {\small $I_\lambda$};
	\node at (0,4.3) {\small $\A^{T_3}$};
\end{tikzpicture}
\caption{The tree $T_3$ and $\id(\A^{T_3})$}\label{fig:tree3}
\end{figure}
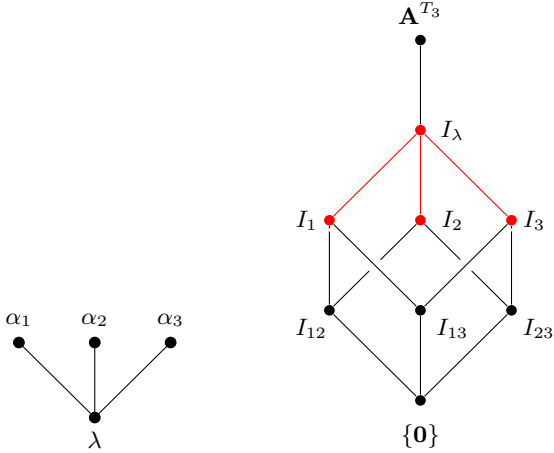 
\end{example}

\begin{example}\label{A^H} Figure \ref{fig:H} shows a tree with height 2; call this tree $H$.
\begin{figure}[h]
\centering
\begin{tikzpicture}

\filldraw (0,0) circle (2pt);
\filldraw (.5,1) circle (2pt);
\filldraw (-.5,1) circle (2pt);
\filldraw (0,2) circle (2pt);
\filldraw (1,2) circle (2pt);

\draw [-] (.5,1) -- (0,0);
\draw [-] (-.5,1) -- (0,0); 
\draw [-] (.5,1) -- (0,2);
\draw [-] (.5,1) -- (1,2);

	\node at (0,-.3) {\small $\lambda$};
	\node at (-.8, 1) {\small $\alpha$};
	\node at (.8, 1) {\small $\beta$};
	\node at (0, 2.3) {\small $\gamma$};
	\node at (1, 2.3) {\small $\delta$};
		
\end{tikzpicture}
\caption{The tree $H$}
\end{figure}
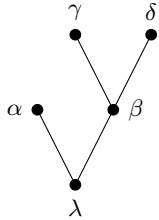 As in the above examples, Theorem \ref{ideals in A^T} tells us that the ideals of $\A^H$ are essentially determined by non-empty subsets of $V(T)$, but in this example we also have $\beta\in[\lambda,\gamma]\cap[\lambda,\delta]$. This reduces some of the possibilities for ideals. Similar to the previous example, write $I_{xy}$ to mean $I_x\cap I_y$. The Hasse diagram  for $\id(\A^H)$ is shown in Figure \ref{fig:idAH}, where again $\X(\A^H)$ is indicated in red.
\begin{figure}
\centering
\begin{tikzpicture}[scale=1.2]
\node[circle,fill=black,inner sep=0pt,minimum size=4pt] (min) at (0,0) {};
\node[circle,fill=black,inner sep=0pt,minimum size=4pt] (b) at (0,1) {};
\node[circle,fill=red,inner sep=0pt,minimum size=4pt] (e) at (0,2) {};
\node[circle,fill=red,inner sep=0pt,minimum size=4pt] (g) at (0,3) {};
\node[circle,fill=black,inner sep=0pt,minimum size=4pt] (a) at (-1,1) {};
\node[circle,fill=black,inner sep=0pt,minimum size=4pt] (c) at (1,1) {};
\node[circle,fill=black,inner sep=0pt,minimum size=4pt] (d) at (-1,2) {};
\node[circle,fill=red,inner sep=0pt,minimum size=4pt] (f) at (1,2) {};

\node[circle,fill=red,inner sep=0pt,minimum size=4pt] (i) at (-1,3) {};
\node[circle,fill=red,inner sep=0pt,minimum size=4pt] (j) at (-.5,3.5) {};
\node[circle,fill=black,inner sep=0pt,minimum size=4pt] (k) at (-.5,4.3) {};

  \draw (d) -- (a) -- (min) -- (b);
  \draw (f) -- (c) -- (min);
  \draw (a) -- (e) -- (c);
  \draw[preaction={draw=white, -,line width=6pt}] (d) -- (b) -- (f);
  \draw[-] (d)--(g);
  \draw[-,red] (g)--(e);
  \draw[-,red] (g)--(f);
  \draw[-,red] (g)--(j);
  \draw[-,red] (i)--(j);
  \draw[-] (j)--(k);
  \draw[-] (i)--(d);
  
	\node at (0,-.4) {\small $\{\0\}$};
	\node at (-1.2, .8) {\small $I_{\alpha\gamma}$};
	\node at (-1.25, 3) {\small $I_{\alpha}$};
	\node at (-.2, 3.5) {\small $I_{\lambda}$};
	\node at (1.3, .8) {\small $I_{\gamma\delta}$};
	\node at (.35,.8) {\small $I_{\alpha\delta}$};
	\node at (-1.35, 2) {\small $I_{\alpha\beta}$};
	\node at (1.25, 2) {\small $I_\delta$};
	\node at (.35, 2) {\small $I_\gamma$};
	\node at (.35,3) {\small $I_\beta$};
	\node at (-.5,4.6) {\small $\A^H$};
\end{tikzpicture}
\caption{The ideals of $\A^H$}\label{fig:idAH}
\end{figure}
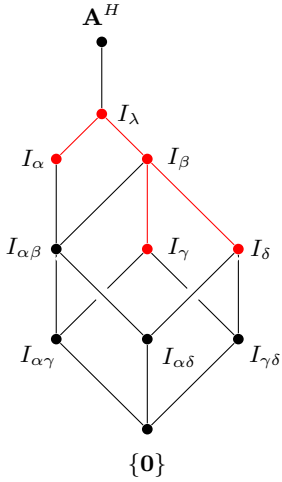

\end{example}


\subsection{The algebras $\A^{T_n}$}

Let $T_n$ denote the rooted tree of height one with $n$ leaves as shown in Figure \ref{fig:tree n}.
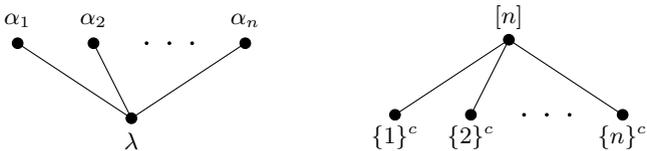
\begin{figure}[b]
\centering
\begin{tikzpicture}

\filldraw (0,0) circle (2pt);
\filldraw (-1.5,1) circle (2pt);
\filldraw (-.5,1) circle (2pt);
\filldraw (.2,1) circle (.5pt);
\filldraw (.5,1) circle (.5pt);
\filldraw (.8,1) circle (.5pt);
\filldraw (1.5,1) circle (2pt);

\draw [-] (-1.5,1) -- (0,0);
\draw [-] (-.5,1) -- (0,0); 
\draw [-] (1.5,1) -- (0,0);

	\node at (0,-.3) {\small $\lambda$};
	\node at (-1.5, 1.3) {\small $\alpha_1$};
	\node at (-.5, 1.3) {\small $\alpha_2$};
	\node at (1.5, 1.3) {\small $\alpha_n$};	
	
\end{tikzpicture}
\hspace{1cm}
\begin{tikzpicture}

\filldraw (0,1) circle (2pt);
\filldraw (-1.5,0) circle (2pt);
\filldraw (-.5,0) circle (2pt);
\filldraw (.2,0) circle (.5pt);
\filldraw (.5,0) circle (.5pt);
\filldraw (.8,0) circle (.5pt);
\filldraw (1.5,0) circle (2pt);

\draw [-] (-1.5,0) -- (0,1);
\draw [-] (-.5,0) -- (0,1); 
\draw [-] (1.5,0) -- (0,1);

	\node at (0,1.3) {\small $[n]$};
	\node at (-1.5, -.3) {\small $\{1\}^c$};
	\node at (-.5, -.3) {\small $\{2\}^c$};
	\node at (1.5, -.3) {\small $\{n\}^c$};	
	
\end{tikzpicture}
\caption{The tree $T_n$ and the poset $\text{MI}(\overline{\bb{B}}_n)$}\label{fig:tree n}
\end{figure} Our objective is to generalize the observations made in the examples of \ref{T_2}. To that end, let $\bb{B}_n$ denote the powerset of $\{1,2,\ldots, n\}$; this is the unique (up to isomorphism) finite Boolean algebra with $n$ atoms. We will let $\overline{\bb{B}}_n$ denote $\bb{B}_n\oplus \1$, which is $\bb{B}_n$ with a new top element $\1$ adjoined. So $S<\1$ for all $S\in \bb{B}_n$. The poset of meet-irreducible elements of $\overline{\bb{B}}_n$ is also shown in Figure \ref{fig:tree n}.

\begin{remark} In the literature, pseudocomplemented lattices are sometimes referred to as p-algebras. Theorem 2 of  Lakser's paper \cite{lakser71} characterizes algebras of the form $\overline{\bb{B}}_n$ for $n\in\bb{N}$ as precisely the finite subdirectly irreducible distributive p-algebras. 
\end{remark}

For a lattice $\D$, let $\text{MI}(\D)$ denote the poset of meet-irreducible elements in $\D$. It is a theorem of Birkhoff that a finite poset $(P,\leq)$ uniquely determines (up to isomorphism) a finite distributive lattice $\D$ such that $P\cong \text{MI}(\D)$ as posets; see \cite{birkhoff67}.

\begin{theorem}\label{ideals of A^T_n form a p-alg} As lattices, $\id(\A^{T_n})\cong \overline{\bb{B}}_n$.
\end{theorem}

\begin{proof}
Since $T_n$ is a finite rooted tree we have $\X(\A^{T_n})\cong T_n^\partial$ as posets by Corollary \ref{X(A^T) = T^d}, and so $\text{MI}(\overline{\bb{B}}_n)\cong T_n^\partial\cong \X(\A^{T_n})$. However, $\X(\A^{T_n})$ is precisely the set of meet-irreducible elements of $\id(\A^T)$, and we know that $\id(\A^T)$ is a distributive lattice. By Birkhoff's theorem we therefore have $\id(\A^{T_n})\cong \overline{\bb{B}}_n$
\end{proof}


\section{Spectra}

In this section we consider spectra of cBCK-algebras. Pa\l asinski first explored a topological representation for cBCK-algebras in \cite{palasinski82}. Hoo and Murty defined the spectrum of a bounded cBCK-algebra in \cite{hoo87}; the topological space they define is not the same as Pa\l asinski's. The spectrum as defined by Hoo and Murty became the standard. Later Meng and Jun proved in \cite{MJ98} that the spectrum of a bounded cBCK-algebra is a spectral space. Aslam, Deeba, and Thaheem studied spectra of cBCK-algebras both with and without the assumption boundedness in \cite{ADT93}. The present work should be viewed as a continuation of the work in \cite{ADT93}, as well as a generalization of the work in \cite{MJ98}. In particular, we will show that the spectrum of any cBCK-algebra is a locally compact generalized spectral space, with compactness if and only if the algebra is finitely generated as an ideal. We go on to show that when the algebra is involutory, whether bounded or not, the spectrum is a Priestley space.

Let $\A$ be a cBCK-algebra. For $S\subseteq\A$, define \[\sigma(S)=\{P\in\X(\A)\mid S\not\subseteq P\}\,.\] We will write $\sigma(a)$ for $\sigma(\{a\})$. It is straightforward to show that, for any $S\subseteq \A$, we have $\sigma(S)=\sigma\bigl((S]\bigr)$. In particular $\sigma(a)=\sigma\bigl((a]\bigr)$. Another computation gives $\sigma(a)\cap\sigma(b) = \sigma(a\meet b)$ for $a,b\in \A$.

The collection \[\mc{T}(\A)=\{\sigma(I)\mid I\in\id(\A)\}\] is a topology on $\X(\A)$, and the set \[\mc{T}_0(\A)=\{\sigma(a)\mid a\in\A\}\] is a basis. There are several proofs of this in the literature, but we point the reader to Proposition 3.1 of \cite{ADT93}.

\begin{remark}\label{sigma is isom} In \cite{ADT93} it is also shown that the map $\sigma:\id(\A)\to\mc{T}(\A)$ is a lattice isomorphism. This gives an alternate proof that $\mt{cBCK}$ is a congruence-distributive variety since we now have $\mc{T}(\A)\cong\id(\A)\cong\con(\A)$ and any topology forms a distributive lattice.\end{remark}

\begin{definition} The space $\bigl(\X(\A)\,,\, \mc{T}(\A)\bigr)$ is the \textit{spectrum} of $\A$.\end{definition}

For a topological space $X$, denote the collection of compact open subsets of $X$  by $\K(X)$. We say a topological space $X$ is \textit{quasi-sober} if every non-empty irreducible closed subset is the closure of a point, and \textit{sober} if that point is unique. It is well-known that a space is sober if and only if it is both $T_0$ and quasi-sober.

A topological space $X$ is called a \textit{spectral space} if $X$ is homeomorphic to the spectrum of some commutative ring. Hochster provided the following characterization of spectral spaces in his PhD thesis.

\begin{theorem}[\cite{hochster69}]\label{hochster} A space $X$ is spectral if and only if the following conditions are satisfied:
\begin{enumerate}
\item[(H1)] $X$ is compact
\item[(H2)] $X$ is $T_0$
\item[(H3)] $\K(X)$ is a basis and closed under finite intersections
\item[(H4)] $X$ is quasi-sober.
\end{enumerate} 
\end{theorem}

\begin{remark} The term \textit{multiplicative basis} is sometimes used when a space satisfies (H3). We will use this terminology.
\end{remark}

\begin{remark} As mentioned at the beginning of this section, Meng and Jun proved in \cite{MJ98} that $\X(\A)$ is a spectral space for any \textit{bounded} cBCK-algebra $\A$. Confusingly, in their paper they use the term ``Stone space'' instead of ``spectral space.'' This use of terminology seems to come from Balbes and Dwinger's text \cite{balbes11}. In the literature, ``Stone space'' typically refers to a topological space which is compact, Hausdorff, and totally disconnected. The spectrum of a commutative ring $R$ -- by definition a spectral space -- is rarely Hausdorff: the closed points in $\text{Spec}(R)$ with respect to the Zariski topology are precisely the maximal ideals of $R$. The same is true for the spectrum of a cBCK-algebra as well. 

Despite this bit of confusion, the theorem of Meng and Jun points at a connection between bounded cBCK-algebras and commutative rings. Namely, if $\A$ is a bounded cBCK-algebra then $\X(\A)\simeq \text{Spec}(R)$ for some commutative ring $R$. In general, constructing such a ring $R$ is a rather complicated process. We refer the reader to Hochster's original 1969 paper \cite{hochster69}, the papers by Lewis \cite{lewis73} and Ershov \cite{ershov05} which discuss alternate constructions in the finite setting, or the very readable thesis by Tedd \cite{tedd16} which compares and generalizes the various constructions.
\end{remark}

Let us note here that $\sigma(a)$ is compact open in $\X(\A)$ for any $a\in \A$. We point to Corollary 4 of \cite{MJ98} for the proof, which does not require $\A$ to be bounded. So if $\A$ is bounded with upper bound 1, say, then $\sigma(1)=\X(\A)$ since any ideal containing 1 cannot be proper. In this case, $\X(\A)$ is compact. 

What happens to the spectrum $\X(\A)$ when $\A$ is not assumed to be bounded? We will see in Example \ref{noncompact example} that compactness can fail. On the other hand boundedness is not necessary for compactness: any finite spectrum $\X(\A)$ is trivially compact whether $\A$ is bounded or not. If nothing else, local compactness is always guaranteed.

\begin{corollary}\label{X is LC} For any cBCK-algebra $\A$, the spectrum $\X(\A)$ is locally compact.
\end{corollary}

\begin{proof} Every prime ideal is contained in $\sigma(a)$ for some $a$ since these sets form a basis, and the $\sigma(a)$'s are compact.
\end{proof}

Here we provide necessary and sufficient conditions for compactness of $\X(\A)$.

\begin{theorem}\label{cpct iff} The space $\X(\A)$ is compact if and only if $\A$ is finitely generated as an ideal. 
\end{theorem}

\begin{proof} Suppose $\X(\A)$ is compact. Since $\mc{T}_0=\{\sigma(a)\mid a\in\A\}$ is a basis for $\X(\A)$, we can write $\X(\A)=\bigcup_{a\in A}\sigma(a)$. By compactness we must have $\X(\A)=\bigcup_{i=1}^k\sigma(u_i)$ for some elements $u_1,\ldots, u_k\in A$. We notice that no prime ideal can contain all the $u_i$'s, for otherwise $P\notin \sigma(u_i)$ for every $i$; this is a contradiction since $\X(\A)=\bigcup_{i=1}^k\sigma(u_i)$. In particular, since maximal ideals are prime, no maximal ideal can contain the set $\{u_1, \ldots, u_k\}$. Thus the smallest ideal containing all of the $u_i$'s is $\A$. That is, $\A=(u_1, \ldots, u_k]$.

Conversely, assume $\A$ is finitely generated as an ideal. Then $\A=(u_1, \ldots, u_k]$ for some elements $u_1, \ldots, u_k\in A$. Because prime ideals are proper, we must have $\{u_1,\ldots, u_k\}\not\subseteq P$ for all $P\in\X(\A)$. This means that for each prime ideal $P$ there is some $i\in\{1,\ldots, k\}$ such that $P\in\sigma(u_i)$, and hence $\X(\A)=\bigcup_{i=1}^k\sigma(u_i)$. We know each $\sigma(u_i)$ is compact, and so $\X(\A)$ is a finite union of compact sets.
\end{proof}

\begin{definition} A topological space $X$ is a \textit{generalized spectral space} if it satisfies (H2)-(H4).\end{definition}

So a generalized spectral space which happens to be compact is a spectral space. We will prove that $\X(\A)$ is a generalized spectral space in steps.

\begin{lemma}\label{X(A) is T_0} $\X(\A)$ is $T_0$.\end{lemma}

\begin{proof} Take $P,Q\in\X(\A)$ with $P\neq Q$. Without loss of generality there is some $a\in P$ such that $a\notin Q$. Then $Q\in\sigma(a)$ and $P\notin \sigma(a)$, so $\sigma(a)$ is an open set separating $P$ and $Q$. 
\end{proof}

\begin{lemma}\label{cpct opens form basis} The compact open sets $\K\X(\A)$ form a multiplicative basis for $\X(\A)$, so $\X(\A)$ satisfies (H3).
\end{lemma}

\begin{proof} We know $\mc{T}_0=\{\sigma(a)\mid a\in \A\}$ is a basis for $\X(\A)$ consisting of compact subsets. Let $\overline{\mc{T}}_0$ be the closure of $\mc{T}_0$ under finite unions; a moment's thought gives $\overline{\mc{T}}_0=\K\X(\A)$, and this is certainly a basis for the same topology. Take $U,V\in\K\X(\A)$. Then 
\begin{align*}
U\cap V=\Bigl(\bigcup_{i=1}^n\sigma(a_i)\Bigr)\cap\Bigl(\bigcup_{j=1}^m\sigma(b_j)\Bigr)
&=\bigcup_{i=1}^n\bigcup_{j=1}^m\sigma(a_i)\cap\sigma(b_j)\\
&=\bigcup_{i=1}^n\bigcup_{j=1}^m\sigma(a_i\meet b_j)\in\K\X(\A)\,,
\end{align*} and thus $\overline{\mc{T}}_0=\K\X(\A)$ is a multiplicative basis.
\end{proof} 

For $I\in\id(\A)$, define $V(I):=\{Q\in\X(\A)\mid I\subseteq Q\}=\sigma(I)^c$. It is straightforward to check that $V$ is order-reversing and that $V(I\cap J)=V(I)\cup V(J)$ for any ideals $I$ and $J$.

\begin{lemma}\label{X(A) is sober} Every irreducible closed set in $\X(\A)$ has the form $V(P)$ for some $P\in\X(\A)$, and $V(P)$ is the closure of $\{P\}$. Consequently, $\X(\A)$ is quasi-sober, satisfying (H4).
\end{lemma}

\begin{proof} Let $P\in\X(\A)$. Note that $V(P)=\sigma(P)^c$, so $V(P)$ is closed. Next, suppose $V(P)=C\cup D$ for two proper closed subsets $C,D \subsetneq V(P)$. Since $C$ and $D$ are closed, there are ideals $I,J\in\id(\A)$ such that $C=\sigma(I)^c$ and $D=\sigma(J)^c$. Thus,
\[\sigma(P)^c=V(P)=\sigma(I)^c\cup\sigma(J)^c=\bigl(\sigma(I)\cap\sigma(J)\bigr)^c=\sigma(I\cap J)^c\] which implies that $\sigma(P)=\sigma(I\cap J)$. But by Remark \ref{sigma is isom}, $\sigma$ is injective, so $P=I\cap J$. Since $P$ is prime it is irreducible, so either $P=I$ or $P=J$. Without loss of generality, assume $P=I$. Then $V(P)=V(I)=\sigma(I)^c=C$, which is a contradiction. Hence, we cannot write $V(P)$ as a union of proper closed subsets, meaning $V(P)$ is irreducible.

On the other hand, assume $C$ is an irreducible closed subset of $\X(\A)$. Since $C$ is closed we have $C=\sigma(I)^c=V(I)$ for some ideal $I$. We claim that $I$ is prime.

Suppose not. Then there are ideals $J_1$ and $J_2$ such that $J_1\cap J_2=I$ but $J_1\neq I$ and $J_2\neq I$. Then \[V(I)=V(J_1\cap J_2)=V(J_1)\cup V(J_2)\,.\] Suppose $V(I)=V(J_1)$. Pa\l asinski proved in \cite{palasinski82(2)} that any ideal is equal to the intersection of the prime ideals containing it; from this we obtain \[I=\bigcap_{P\in V(I)}P=\bigcap_{Q\in V(J_1)}Q=J_1\,.\] This is a contradiction, and a similar analysis holds for $J_2$. But $J_1\cap J_2\subseteq J_1$ means $V(J_1)\subseteq V(J_1\cap J_2)=V(I)$, so we must have $V(J_1)\subsetneq V(I)$. Similarly, $V(J_2)\subsetneq V(I)$. Hence, if $I$ is not prime, we can decompose $C=V(I)=V(J_1)\cup V(J_2)$ into a union of proper closed subsets, a contradiction. Therefore, any irreducible closed subset has the form $V(P)$ for some $P\in\X(\A)$. 

Lastly we show that $V(P)=\overline{\{P\}}$. Suppose $P\in C$ where $C$ is a closed set; there is some ideal $I$ such that $C=\sigma(I)^c=V(I)$. So $P\in V(I)$, meaning $I\subseteq P$, and thus $V(P)\subseteq V(I)=C$. Hence, $V(P)$ is the smallest closed set containing $P$, meaning $V(P)=\overline{\{P\}}$. 
\end{proof} 

\begin{theorem}\label{X(A) is LC gspec} For any cBCK-algebra $\A$, the spectrum $\X(\A)$ is a locally compact generalized spectral space.\end{theorem}

\begin{proof}
Combine Corollary \ref{X is LC} and Lemmas \ref{X(A) is T_0}, \ref{cpct opens form basis}, and \ref{X(A) is sober}.
\end{proof}

We recall that a \textit{Priestley space} is an ordered topological space $(X,\leq, \mc{T})$ that is compact and satisfies the following separation property (PSA): if $x\not\leq y$, there exists a clopen up-set $U$ such that $x\in U$ but $y\notin U$. Notice that (PSA) implies that $X$ is Hausdorff. It is a theorem (see Theorem 4.2 of \cite{johnstone82}) that the following conditions are equivalent:
\begin{enumerate}
\item $X$ is Hausdorff, sober, and $\K(X)$ is a multiplicative basis.
\item $X$ is compact, Hausdorff, and totally disconnected (that is, $X$ is a Stone space).
\end{enumerate} Thus, if a space $X$ satisfies (PSA), is sober, and $\K(X)$ is a multiplicative basis, then $X$ is a Priestley space.

In \cite{ADT93}, the authors showed that if $\A$ is involutory then $P\in\sigma(I)$ if and only if $P\notin\sigma(I^\ast)$ for any $I\in\id(\A)$. Said another way, if $\A$ is involutory, then $\sigma(I)=V(I^\ast)=\sigma(I^\ast)^c$ for any ideal $I$ of $\A$.

\begin{lemma}\label{invol imply sigma is clopen upset} Let $\A$ be involutory. Then $\sigma(I)$ is a clopen up-set in $\X(\A)$ for all $I\in\id(\A)$.
\end{lemma}

\begin{proof} Take $I\in\id(\A)$. Then $\sigma(I^\ast)^c=\sigma(I)$. Since $I^\ast$ is an ideal, $\sigma(I^\ast)$ is an open set and hence $\sigma(I)$ is clopen. 

Now take $P, Q\in\X(\A)$ with $P\subseteq Q$. Assume $P\in\sigma(I)$. Then
\begin{align*}
P\notin\sigma(I^\ast) &\Longrightarrow I^\ast\subseteq P\subseteq Q\\
&\Longrightarrow Q\notin\sigma(I^\ast)\\
&\Longrightarrow Q\in\sigma(I)
\end{align*} which shows that $\sigma(I)$ is an up-set.
\end{proof}

\begin{theorem}\label{invol implies pries} If $\A$ is involutory, then $\X(\A)$ is a Priestley space.
\end{theorem}

\begin{proof} We first show $\X(\A)$ satisfies (PSA). Suppose we have $P,Q\in\X(\A)$ with $P\not\subseteq Q$. Then there is some $a\in \A$ such that $a\in P$ and $a\notin Q$, which implies $P\notin \sigma(a)$ and $Q\in\sigma(a)$. Therefore we have $P\in\sigma\bigl(\{a\}^\ast\bigr)$ and $Q\notin\sigma\bigl(\{a\}^\ast\bigr)$. Lemma \ref{invol imply sigma is clopen upset} tells us $\sigma\bigl(\{a\}^\ast\bigr)$ is a clopen up-set. Hence, $\X(\A)$ satisfies (PSA).

We have already seen in Lemmas \ref{X(A) is sober} and \ref{cpct opens form basis} that $\X(\A)$ is sober and that $\K\X(\A)$ is a multiplicative basis. So $\X(\A)$ is a Priestley space.
\end{proof}

We now collect some properties of spectra for any cBCK-algebra $\A$. Let $\mrm{M}(\A)$ denote the maximal spectrum of $\A$; that is, the subset of $\X(\A)$ consisting of maximal ideals, endowed with the subspace topology. 

A cBCK-algebra $\A$ is \textit{directed} if each pair of elements has an upper bound.

\begin{proposition}\label{small properties} Let $\A$ be a cBCK-algebra.
\begin{enumerate}
\item $\X(\A)$ is Hausdorff if and only if each $\sigma(a)$ is clopen.

\item If each $\sigma(a)$ is clopen, then $\X(\A)$ is zero-dimensional, totally disconnected, and completely regular.

\item A point $\{M\}\in\X(\A)$ is closed if and only if $M$ is a maximal ideal.

\item If every prime ideal of $\A$ is maximal, then $\X(\A)$ is Hausdorff

\item If $\A$ is directed, then $\mrm{M}(\A)$ is Hausdorff.

\item If $\A$ is directed, then $\X(\A)$ is Hausdorff if and only if $\X(\A)=\mrm{M}(\A)$.

\end{enumerate}
\end{proposition}

\begin{proof}(1) Assume each $\sigma(a)$ is clopen. Take $P,Q\in\X(\A)$ with $P\neq Q$. Without loss of generality, there is some $x\in P$ such that $x\notin Q$. Then $Q\in\sigma(x)$ while $P\in\sigma(x)^c$. The sets $\sigma(x)$ and $\sigma(x)^c$ are obviously disjoint and open, so $\X(\A)$ is Hausdorff.

On the other hand, assume $\X(\A)$ is Hausdorff and take $a\in \A$. We know $\sigma(a)$ is compact open, but compact subsets of Hausdorff spaces are closed. So $\sigma(a)$ is clopen.

(2) Assume each $\sigma(a)$ is clopen. Then $\mc{T}_0$ is a basis of clopen sets which means $\X(\A)$ is zero-dimensional. From Lemma \ref{X(A) is T_0} and (1) above, we see that $\X(\A)$ is locally compact Hausdorff. Any locally compact Hausdorff space is completely regular. Lastly, for locally compact Hausdorff spaces, being zero-dimensional is equivalent to being totally disconnected.

(3) We know maximal ideals are prime, and we saw in Lemma \ref{X(A) is sober} that $\overline{\{P\}}=V(P)=\{Q\in\X(\A)\mid P\subseteq Q\}$ for any prime ideal $P$. Thus, if $M$ is a maximal ideal, we have $\overline{\{M\}}=\{M\}$ by maximality.

On the other hand, suppose $\{P\}\subseteq \X(\A)$ is closed. We claim that $P$ is a maximal ideal. To see this, suppose there is an ideal $M$ with $P\subseteq M$ and let $C$ be a closed set containing $P$. Then $C=\sigma(J)^c$ for some ideal $J$, and so $P\in \sigma(J)^c$. Thus, $J\subseteq P\subseteq M$, meaning $M\in\sigma(J)^c=C$ as well. But this implies $M\in\bigcap \{C\subseteq\X(\A)\mid P\in C, \text{ $C$ is closed}\}=\{P\}$. Thus, $P=M$ and $P$ is maximal.

(4) Assume $\X(\A)=\mrm{M}(\A)$. By (3), then, it follows that every point is closed and hence $\X(\A)$ is Hausdorff.

(5) Assume $\A$ is directed. Then $(x\cdot y)\meet(y\bdot x)=0$ for all $x,y\in \A$ (see Lemma 5.2.2 and Theorem 5.2.28 of \cite{DP00}). Now take $M_1,M_2\in\mrm{M}(\A)$ with $M_1\neq M_2$. By maximality we have $M_1\not\subseteq M_2$ and $M_2\not\subseteq M_1$. Pick $a\in M_1\setminus M_2$ and $b\in M_2\setminus M_1$. We know $a\bdot b\leq a$ and $b\bdot a\leq b$, and since ideals are down-sets this means $a\bdot b\in M_1$ and $b\bdot a\in M_2$. If $a\bdot b\in M_2$, then we must have $a\in M_2$, a contradiction. So $a\bdot b\in M_1\setminus M_2$, and similarly $b\bdot a\in M_2\setminus M_1$. Thus, $M_1\in \sigma(b\bdot a)$ and $M_2\in \sigma(a\bdot b)$. But notice that 
\[\sigma(a\bdot b)\cap \sigma(b\bdot a)=\sigma\bigl((a\bdot b)\meet(b\bdot a)\bigr)=\sigma(0)=\emptyset\,,\] so $\sigma(a\bdot b)$ and $\sigma(b\bdot a)$ are disjoint open sets separating $M_1$ and $M_2$. Hence, $\mrm{M}(\A)$ is Hausdorff.

(6) This follows from the fact that $(x\cdot y)\meet(y\bdot x)=0$ for all $x,y\in \A$ together with Theorem 6.1.7 of \cite{DP00}.
\end{proof}

We close this section with a theorem that will allow us to more effectively compute spectra for certain algebras, as well as one example computation to illustrate the idea.

\begin{theorem}\label{X(U) is coproduct} Let $\U=\bigcupdot_{\lambda\in\Lambda}\A_\lambda$, where $\Lambda$ is any indexing set. Then $\mrm{X}(\U)$ is homeomorphic to $\bigsqcup_{\lambda\in\Lambda}\mrm{X}(\A_\lambda)$ with the disjoint union topology. That is,
\[\mrm{X}\Bigl(\bigcupdot_{\lambda\in\Lambda}\A_\lambda\Bigr)\simeq\bigsqcup_{\lambda\in\Lambda}\mrm{X}(\A_\lambda)\,.\]
\end{theorem}

\begin{proof} For any $P\in\bigsqcup_{\lambda\in\Lambda}\mrm{X}(\A_\lambda)$, we know $P$ is a prime ideal of $\A_\mu$ for some $\mu$. Define \[\Phi\colon\bigsqcup_{\lambda\in\Lambda}\mrm{X}(\A_\lambda)\to\mrm{X}\Bigl(\bigcupdot_{\lambda\in\Lambda}\A_\lambda\Bigr)\] by $\Phi(P)=\bigcupdot_{\lambda\in\Lambda}\A_{\lambda,\mu}^P$. By Theorem \ref{primes_in_union} this map is surjective. If $\Phi(P)=\Phi(Q)$, then $\bigcupdot_{\lambda\in\Lambda}\A_{\lambda,\alpha}^P=\bigcupdot_{\lambda\in\Lambda}\A_{\lambda,\beta}^Q$, where $P\in\X(\A_\alpha)$ and $Q\in\X(\A_\beta)$. This is only possible if $\alpha=\beta$ and $P=Q$. So $\Phi$ is a bijection.

We now show that $\Phi$ is continuous. Let $\sigma_\U(a)$, for $a\in \U$, be a basic open set in $\mrm{X}(\U)$. Since $a\in \U$, we have $a\in\A_\beta$ for some $\beta\in\Lambda$. We claim that $\Phi^{-1}\bigl(\sigma_\U(a)\bigr)=\sigma_{\A_\beta}(a)$.

Take $P\in\Phi^{-1}\bigl(\sigma_\U(a)\bigr)$, so $\Phi(P)\in\sigma_\U(a)$. By definition of $\Phi$ we have $\Phi(P)=\bigcupdot_{\lambda\in\Lambda}\A_{\lambda,\mu}^P$ for some index $\mu$. Since $a\in\A_\beta$ but $a\notin \Phi(P)=\bigcupdot_{\lambda\in\Lambda}\A_{\lambda,\mu}^P$, it follows that $\mu=\beta$; that is, $P\in\mrm{X}(\A_\beta)$ and $\Phi(P)=\bigcupdot_{\lambda\in\Lambda}\A_{\lambda,\beta}^P$. So then $a\notin P$ which means $P\in\sigma_{\A_\beta}(a)$.

For the other inclusion, take $P\in\sigma_{\A_\beta}(a)$. So $P\in\mrm{X}(\A_\beta)$ and $a\notin P$. Then $\Phi(P)=\bigcupdot_{\lambda\in\Lambda}\A_{\lambda,\beta}^P$ and $a\notin \Phi(P)$ as well, so $\Phi(P)\in\sigma_\U(a)$. Hence $P\in\Phi^{-1}\bigl(\sigma_\U(a)\bigr)$, and therefore $\Phi^{-1}\bigl(\sigma_\U(a)\bigr)=\sigma_{\A_\beta}(a)$ as claimed. 

Next we note that $\sigma_{\A_\beta}(a)$ is open in the disjoint union topology on $\bigsqcup_{\lambda\in\Lambda}\mrm{X}(\A_\lambda)$. To see this, notice
\[\sigma_{\A_\beta}(a)\cap\mrm{X}(\A_\lambda)=\left.\begin{cases}\emptyset & \text{if $\beta\neq\lambda$}\\\sigma_{\A_\beta}(a) & \text{if $\beta=\lambda$}\end{cases}\right\}\,,\] meaning $\sigma_{\A_\beta}(a)\cap\mrm{X}(\A_\lambda)$ is open in $\mrm{X}(\A_\lambda)$ for all $\lambda$, and therefore $\sigma_{\A_\beta}(a)$ is open in the disjoint union topology. Thus, the preimage under $\Phi$ of any basic open set of $\X(\U)$ is open in $\bigsqcup_{\lambda\in\Lambda}\mrm{X}(\A_\lambda)$, meaning $\Phi$ is continuous.

We show that $\Phi$ is an open map. Let $V\subseteq \bigsqcup_{\lambda\in\Lambda}\mrm{X}(\A_\lambda)$ be open. Then $V\cap\mrm{X}(\A_\lambda)$ is open in $\mrm{X}(\A_\lambda)$ for each $\lambda$. Thus, for each $\lambda$, we have $V\cap\mrm{X}(\A_\lambda)=\sigma_{\A_\lambda}(I_\lambda)$ for some $I_\lambda\in\id(\A_\lambda)$. Put $I=\bigcupdot_{\lambda\in\Lambda}I_\lambda$. We will prove that $\Phi(V)=\sigma_\U(I)$.

Take $Q\in\Phi(V)$; so $Q$ is a prime ideal in $\bigcupdot_{\lambda\in\Lambda}\A_\lambda$. Thus, $Q=\bigcupdot_{\lambda\in\Lambda}\A_{\lambda,\mu}^P$ for some index $\mu$ and some $P\in\mrm{X}(\A_\mu)$. So $\Phi(P)=Q$ meaning $P\in V$, and so $P\in V\cap\mrm{X}(\A_\mu)=\sigma_{\A_\mu}(I_\mu)$. Thus, $I_\mu\not\subseteq P$ which implies that $I=\bigcupdot_{\lambda\in\Lambda}I_\lambda\not\subseteq \bigcupdot_{\lambda\in\Lambda}\A_{\lambda,\mu}^P=Q$. Hence $Q\in\sigma_\U(I)$.

On the other hand, take $Q\in\sigma_\U(I)$. Then $I\not\subseteq Q$ and $Q=\bigcupdot_{\lambda\in\Lambda}\A_{\lambda,\mu}^P$ for some $\mu$ and some $P\in\mrm{X}(\A_\mu)$. It follows that $I_\mu\not\subseteq P$, and so $P\in\sigma_{\A_\mu}(I_\mu)=V\cap\mrm{X}(\A_\mu)$. In particular, $P\in V$ and $\Phi(P)=Q$, so $Q\in\Phi(V)$. 

Therefore, $\Phi(V)=\sigma_\U(I)$ which tells us $\Phi$ is open map. Since $\Phi$ is an open continuous bijection, it is a homeomorphism.
\end{proof}

Let $\A$ be a simple cBCK-algebra. Then $\A$ has exactly two ideals; the trivial ideal is automatically maximal, and therefore prime. Thus, for any simple cBCK-algebra the spectrum $\X(\A)$ is a one-point space, and so the spectra of any two simple cBCK-algebras are homeomorphic. For example, $\X(\C_1)\simeq\X(\N_0)$, despite the algebras $\C_1$ and $\N_0$ being different order types!

\begin{example}\label{noncompact example} Consider $\U=\bigcupdot_{\lambda\in\Lambda}\C_1$. By Theorem \ref{X(U) is coproduct} we have $\X(\U)\simeq \bigsqcup_{\lambda\in\Lambda}\X(\C_1)$ with the disjoint union topology. But $\X(\C_1)$ is a one-point space, so $\X(\U)$ is a discrete space with cardinality $|\Lambda|$. Thus, $\X(\U)$ is not compact unless $\Lambda$ is finite, and a subset $V$ of $\X(\U)$ is compact if and only if $V$ is finite.

Let us label the atoms of $\U$ by $\{a_\lambda\}_{\lambda\in \Lambda}$. Applying Theorem \ref{primes_in_union}, every prime ideal of $\U$ is of the form $P_\lambda=\U\setminus\{a_\lambda\}$ and the basis for our topology is 
\[\mc{T}_0=\bigl\{\,\sigma(0)\,\bigr\}\cup\bigl\{\,\sigma(a_\lambda)\,\bigr\}_{\lambda\in \Lambda}=\bigl\{\,\emptyset\,\bigr\}\cup\bigl\{\,\{P_\lambda\}\,\bigr\}_{\lambda\in \Lambda}\,.\] Hence, $\mc{T}(\U)$ is lattice-isomorphic to $\mt{P}(\Lambda)$, the powerset of $\Lambda$, and $\K\X\bigl(\U)$ is lattice-isomorphic to $\mt{P}_{\text{fin}}(\Lambda)$, the lattice of finite subsets of $\Lambda$.

In particular, we have $\K\X\bigl(\bigcupdot_{i=1}^n\C_1\bigr)$ is lattice-isomorphic to $\bb{B}_n$, the finite Boolean algebra of order $2^n$.

\end{example}


\section{Functoriality of $\K$ and $\X$}

Let $X$ and $Y$ be generalized spectral spaces. A map $g\colon X\to Y$ is a \textit{spectral map} if the inverse image of every compact open subset of $Y$ is compact open in $X$. That is, $g^{-1}\bigl(\K(Y)\bigr)\subseteq \K(X)$. Since $\K(Y)$ forms a basis for the topology on $Y$, any spectral map is continuous. Let $\mt{GSpec}$ denote the category of generalized spectral spaces with spectral maps as morphisms. Similarly, let $\mt{Spec}$ denote the category of spectral spaces with spectral maps as morphisms.

We have already seen that the spectrum of any cBCK-algebra is a generalized spectral space. That is, $\X(\A)\in \mt{GSpec}$ for any $\A\in\mt{cBCK}$. Suppose $f\colon \A\to \B$ is a BCK-homomorphism. For any prime ideal $Q$ in $\B$, the preimage $f^{-1}(Q)$ is a prime ideal in $\A$. So we define $\X(f)\colon \X(\B)\to \X(\A)$ by $\X(f)(Q)=f^{-1}(Q)$ for $Q\in \X(\B)$. It is straightforward to check that $\X:\mt{cBCK}\to \mt{GSpec}$ is a contravariant functor; for a proof see Proposition 4.1.1 of \cite{evans20}

This functor cannot be fully faithful since fully faithful functors are injective on objects. We saw earlier that $\X(\C_1)\simeq \X(\N_0)$ -- they are both one-point spaces -- but certainly $\C_1\not\cong \N_0$.

This has the further implication that $\X$ does not yield a dual equivalence of categories. In this way, our situation is similar to that of commutative rings. The functor $\text{Spec}\colon\mt{CommRing}\to\mt{Spec}$ which associates a commutative ring to its prime spectrum is also not a dual equivalence. We contrast this with the well-known dual equivalences between Boolean algebras and Stone space, $\mt{BA}\cong^\partial\mt{Stone}$, or between bounded distributive lattices and Priestley spaces, $\mt{BDL}\cong^\partial \mt{Pries}$ (see \cite{stone36} and \cite{pries70}). 

It is also well-known that the category of Priestley spaces is equivalent to (in fact, isomorphic to) the category of spectral spaces (see \cite{cornish75}). So by the preceding paragraph we have $\mt{BDL}\cong^\partial\mt{Pries}\cong\mt{Spec}$. Hence, there is a dual equivalence $\mt{BDL}\cong^\partial\mt{Spec}$ between bound\-ed distributive lattices and spectral spaces. This duality extends to a duality between the category $\mt{DL}_0$ of distributive lattices with 0, where the morphisms are 0-preserving lattice homomorphisms with cofinal range, and the category $\mt{GSpec}$ of generalized spectral spaces (see \cite{stone38}). This duality sends a distributive lattice $\D$ to its prime spectrum $\text{Spec}(\D)$ endowed with the Zariski topology in one direction, and it sends a generalized spectral space $Y$ to the lattice $\K(Y)$ of compact open subsets in the other. Our situation is diagrammed below.
\begin{center}
\begin{tikzcd}
\mt{cBCK} \arrow[r, "\X"] & \mt{GSpec} \arrow[r, "\K", bend left] & \mt{DL}_0 \arrow[l, "\text{Spec}", bend left]
\end{tikzcd}
\end{center} It would be very nice to have an explicit characterization of the image of $\X$ in $\mt{GSpec}$, but this is a difficult problem. On the other hand, the dual equivalence between $\mt{DL}_0$ and $\mt{GSpec}$ (as well as $\mt{BDL}$ and $\mt{Spec}$) has been studied and may be a fruitful way of gaining leverage on the situation. In particular, it would be interesting to know what distributive lattices lie in the image of the composite functor $\K\X$. We give partial results in Theorem \ref{culmination}, Corollary \ref{KX yields products}, and Corollary \ref{specific products}

For a topological space $X$ we use the notation $\mc{T}_X$ for the lattice of open sets. It is known (see \cite{CGL99}, Proposition 1.2) that two generalized spectral spaces $X$ and $Y$ are homeomorphic if and only if the lattices $\mc{T}_X$ and $\mc{T}_Y$ are isomorphic.

\begin{corollary}\label{homeo criterion} Let $\A$ be a cBCK-algebra and $Y$ a generalized spectral space. Then $Y\simeq \X(\A)$ if and only if $\mc{T}_Y\cong\id(\A)$ as lattices, if and only if $\K(Y)\cong \K\X(\A)$.
\end{corollary}

\begin{proof} This follows from the preceding paragraph together with Remark \ref{sigma is isom} which tells us $\mc{T}_{\X(\A)}=\mc{T}(\A)\cong\id(\A)$ as lattices. The second equivalence follows because the compact open sets form a basis for the topology.
\end{proof}

While this does give a small inroad for understanding the image of $\X$ in $\mt{GSpec}$, it represents more a change of perspective rather than a reduction in difficulty. For now we will focus our attention on a particular class of topological spaces.

\subsection{Noetherian spaces}

\begin{definition} A topological space $Y$ is \textit{Noetherian} if it satisfies the descending chain condition on closed subsets: for any sequence $C_1\supseteq C_2\supseteq\cdots$ of closed subsets of $Y$, there is some $n\in\bb{N}$ such that $C_{n+k}=C_{n}$ for all $k>0$.
\end{definition}

This is equivalent to saying that $Y$ satisfies the ascending chain condition on open subsets. One can show (see \cite{hartshorne77}, Exercise 2.13) that a space $Y$ is Noetherian if and only if every open set is compact. We also note that any finite topological space is obviously Noetherian.

\begin{proposition}\label{these spectra are Noetherian} If $T$ is a finite rooted tree, the space $\X(\A^T)$ is Noetherian. For any $n\in\bb{N}$, the space $\X\bigl(\bigcupdot_{i=1}^n\C_1\bigr)$ is Noetherian.
\end{proposition}

\begin{proof} These spaces are finite.\end{proof}

\begin{theorem}\label{infinite chain noetherian} The space $\X(\A^{\ch_\infty})$ is Noetherian.
\end{theorem}

\begin{proof} Let $V$ be an open set and let $\mc{U}=\{U_j\}_{j\in J}$ be an open cover of $V$. Since $\bb{P}(\ch_\infty)$ is linearly ordered, every open set of $\X(\A^{\ch_\infty})$ has the form $\sigma\bigl(I(p')\bigr)$ for some root-based path $p'$. Thus there is a root-based path $p$ such that $V=\sigma\bigl(I(p)\bigr)$, and for each $j$ there is a root-based path $q_j$ such that $U_j=\sigma\bigl(I(q_j)\bigr)$. Let $q$ be the shortest-length path among the $q_j$'s. Then $q\subseteq q_j$ for all $j\in J$, which means $I(q_j)\subseteq I(q)$ for all $j\in J$. Hence, $V\subseteq \bigcup_{j\in J}U_j\subseteq I(q)$, and therefore $V$ is compact. Since every open set is compact, $\X(\A^{\ch_\infty})$ is Noetherian.
\end{proof}

The usefulness of an algebra having a Noetherian spectrum is the following:

\begin{lemma}\label{KX(A) = id(A)} If $\A\in\mt{cBCK}$ is such that $\X(\A)$ is Noetherian, then $\K\X(\A)\cong \id(\A)$. 
\end{lemma}

\begin{proof}  Suppose $\X(\A)$ is Noetherian. Then every open set is compact; that is $\K\X(\A)=\mc{T}(\A)$, but then $\K\X(\A)=\mc{T}(\A)\cong \id(\A)$ by Remark \ref{sigma is isom}.
\end{proof}

Thus, under the right circumstances, we can find lattices in the image $\K\X$ by finding lattices that occur as the lattice of ideals of a cBCK-algebra. Of course the assumption that $\X(\A)$ be Noetherian is rather strong, and not all spectra are Noetherian. For example, we saw that $\X\bigl(\bigcupdot_{\lambda\in \Lambda}\C_1\bigr)$ is discrete with cardinality $|\Lambda|$, so if $\Lambda$ is infinite this spectrum is not compact and cannot be Noetherian. Nevertheless, the above lemma is still a useful tool.

\begin{theorem}\label{culmination} The following all lie in the image of $\K\X$:
\begin{enumerate}
\item every distributive lattice $\D$ such that $\text{MI}(\D)\cong T^\partial$, as posets, for some finite rooted tree $T$.
\item every finite chain,
\item any countably infinite chain isomorphic to $(\bb{N}_0^\infty)^\partial$,
\item the underlying lattice of every finite subdirectly irreducible distributive p-algebra,
\item the underlying lattice of every finite Boolean algebra.
\end{enumerate} 
\end{theorem}

\begin{proof}\hfill

(1) This proof follows the same strategy as the proof of Theorem \ref{ideals of A^T_n form a p-alg}. Let $\D$ be such a distributive lattice. By Birkhoff's theorem, this lattice is unique up to isomorphism. Since $T$ is finite so too is $\X(\A^T)$, and we already know that $\X(\A^T)\cong T^\partial$ as posets by Corollary \ref{X(A^T) = T^d} so we have $\text{MI}(\D)\cong T^\partial\cong \X(\A^T)$. But we know that $\X(\A^T)=\text{MI}\bigl(\id(\A^T)\bigr)$, and we know $\id(\A^T)$ is a distributive lattice. Hence, $\D\cong \id(\A^T)$ as lattices by the uniqueness of $\D$. 

Further, because $\X(\A^T)$ is finite it is Noetherian as a topological space. Applying Lemma \ref{KX(A) = id(A)} gives $\D\cong \id(\A^T)\cong \K\X(\A^T)$.

(2) This is a special case of (1) with $\D\cong \n$, the $n$-element chain. In this case, $\text{MI}(\n)$ is a chain with $n-1$ elements, which we view as a rooted tree. We therefore have $\n$ in the image of $\K\X$. This could also be seen using Example \ref{chain of length n} together with the fact that $\X(\A^{\ch_{n-1}})$ is Noetherian.

(3) Theorem \ref{infinite chain noetherian} shows $\X(\A^{\ch_\infty})$ is Noetherian, and therefore $\K\X(\A^{\ch_\infty})\cong\id(\A^{\ch_\infty})\cong (\bb{N}_0^\infty)^\partial$, where the last isomorphism was shown in Example \ref{countable chain}.

(4) Combine Theorem \ref{ideals of A^T_n form a p-alg}, Proposition \ref{these spectra are Noetherian}, and Lemma \ref{KX(A) = id(A)}.

(5) Combine Proposition \ref{these spectra are Noetherian} and Lemma \ref{KX(A) = id(A)}.
\end{proof}

\begin{remark} We note that the process used in Theorem \ref{culmination}(1) will not yield every finite distributive lattice. Consider the free bounded distributive lattice on two generators $\F_2$ shown in Figure \ref{fig:F2}, where $\text{MI}(\F_2)$ is indicated in red. Since the poset of meet-irreducibles is not connected, it does not form a tree. Therefore $\F_2$ cannot be obtained as $\id(\A^T)$ for any finite tree $T$.
\begin{figure}[h]
\centering
\begin{tikzpicture}

\filldraw[red] (0,0) circle (2pt);
\filldraw (0,1) circle (2pt);
\filldraw[red] (-1,2) circle (2pt);
\filldraw[red] (1,2) circle (2pt);
\filldraw[red] (0,3) circle (2pt);
\filldraw (0,4) circle (2pt);

\draw [-] (0,0) -- (0,1);
\draw [-] (0,1) -- (1,2); 
\draw [-] (0,1) -- (-1,2);
\draw [-,red] (1,2) -- (0,3);
\draw [-,red] (-1,2) -- (0,3);
\draw [-] (0,3) -- (0,4);

	\node at (0,-.3) {\small $0$};
	\node at (.6, .9) {\small $x\meet y$};
	\node at (-1.3, 2) {\small $x$};
	\node at (1.3, 2) {\small $y$};	
	\node at (.6, 3.1) {\small $x\join y$};
	\node at (0, 4.3) {\small $1$};
		
\end{tikzpicture}
\caption{}\label{fig:F2}
\end{figure} 
\end{remark}


\subsection{Disjoint union in $\mt{GSpec}$}\label{disjoint union in gspec}

Recall that the functor $\K\colon \mt{GSpec}\to\mt{DL}_0$ provides a dual equivalence. From this it follows that $\K$ sends coproducts to products and vice versa. Given a family of generalized spectral spaces $\{X_\lambda\}_{\lambda\in\Lambda}$, the disjoint union $\bigsqcup_{\lambda\in\Lambda}X_\lambda$ with the disjoint union topology is the coproduct in the category $\mt{Top}$ of all topological spaces. Unfortunately it may not be the coproduct in $\mt{GSpec}$, as the next example shows.

\begin{example} Suppose $\Lambda$ is an infinite indexing set, each $X_\lambda=\{\ast\}$, the one-point space, and put $Z=\{\ast\}$ as well. The one-point space is a spectral space while the disjoint union $\mc{X}:=\bigsqcup_{\lambda\in\Lambda}X_\lambda$ is a generalized spectral space, see Theorem \ref{disjoint union}. For each $\lambda$ we have a unique spectral map $f_\lambda: X_\lambda\to Z$, and the inclusion maps $\text{incl}_\lambda: X_\lambda\to\mc{X}$ are spectral maps as well. Now consider the following diagram:
\begin{center}
\begin{tikzcd}
                                                             &  & Z                                                                                    \\
                                                             &  &                                                                                      \\
X_\lambda \arrow[rr, "\text{incl}_\lambda"'] \arrow[rruu, "f_\lambda"] &  & \mc{X} \arrow[uu, "{\exists !\, f\, ?}"', dotted]
\end{tikzcd}
\end{center}
There is exactly one set map $f$ which makes this diagram commute for all $\lambda$, which is $f(x)=\ast$ for all $x\in\mc{X}$. But this map is not a spectral map since $f^{-1}(\{\ast\})=\mc{X}$, which is not compact since $\Lambda$ is infinite. That is, $f$ is not a morphism in $\mt{GSpec}$. In fact, $\hom_{\mt{GSpec}}(\mc{X}, Z)$ is empty! Hence, $\mc{X}$ is not the coproduct of the $X_\lambda$'s in $\mt{GSpec}$.
\end{example}

However, we will see that for \textit{finite} families in $\mt{GSpec}$, the disjoint union is indeed the coproduct. To do this, we first need to know that the disjoint union of generalized spectral spaces is again generalized spectral. We break the proof into some smaller lemmas.

Let $\{X_\lambda\}_{\lambda\in\Lambda}$ be a family of generalized spectral spaces and put $\mc{X}=\bigsqcup_{\lambda\in\Lambda}X_\lambda$, endowed with the disjoint union topology. Open sets in $\mc{X}$ are of the form $\bigsqcup_{\lambda\in\Lambda}U_\lambda$, where $U_\lambda$ is open in $X_\lambda$. We collect here several useful observations; we refer the reader to the text \cite{DST2019} by Dickmann, Schwartz, and Tressl. 
\begin{enumerate}
\item For each $\lambda$, any open set of $X_\lambda$ is open in $\mc{X}$.

\item A subset $C\subseteq\mc{X}$ is closed in $\mc{X}$ if and only if $C=\bigsqcup_{\lambda\in\Lambda} C_\lambda$, where each $C_\lambda$ is closed in $X_\lambda$. Consequently, for each $\lambda$, any closed subset of $X_\lambda$ is closed in $\mc{X}$.

\item A non-empty subset $C\subseteq\mc{X}$ is irreducible in $\mc{X}$ if and only if there is some index $\lambda\in\Lambda$ such that $C\subseteq X_\lambda$ and $C$ is irreducible in $X_\lambda$.

\item  The compact open subsets of $\mc{X}$ are of the form $\bigsqcup_{\lambda\in F} V_{\lambda}$, where $V_{\lambda}\in\K(X_{\lambda})$ and $F$ is a finite subset of $\Lambda$. Consequently, for each $\lambda$, any compact open subset of $X_\lambda$ is also compact open in $\mc{X}$.

\item A disjoint union of $T_0$ spaces is $T_0$.

\end{enumerate}

\begin{lemma}\label{basis} The compact open subsets of $\mc{X}$ are a multiplicative basis.
\end{lemma}

\begin{proof} Let $U=\bigsqcup_{\lambda\in\Lambda}U_\lambda$ be open in $\mc{X}$. Since $\K(X_\lambda)$ is basis for $X_\lambda$, each $U_\lambda$ can be written as a union of elements in $\K(X_\lambda)$, which are elements of $\K(\mc{X})$ by observation (1). Thus we can write $U$ as a union of elements in $\K(\mc{X})$, meaning $\K(\mc{X})$ is a basis for the disjoint union topology.

Now take $U,V\in\K(\mc{X})$. By observation (4) we can write $U=\bigsqcup_{\lambda\in F}U_\lambda$ and $V=\bigsqcup_{\mu\in G}V_\mu$ where $F$ and $G$ are finite subsets of $\Lambda$, each $U_\lambda$ is compact open in $X_\lambda$, and each $V_\mu$ is compact open in $X_\mu$. If $\lambda\neq \mu$, then $U_\lambda\cap V_\mu=\emptyset$, and so
\begin{align*}
U\cap V
=\Bigl(\bigsqcup_{\lambda\in F}U_\lambda\Bigr)\cap\Bigl(\bigsqcup_{\mu\in G}V_\mu\Bigr)
&=\bigsqcup_{\lambda\in F}\bigsqcup_{\mu\in G}(U_\lambda\cap V_\mu)\\
&=\bigsqcup_{\alpha\in F\cap G}(U_\alpha\cap V_\alpha)\,.
\end{align*} For each $\alpha\in F\cap G$ we know $U_\alpha\cap V_\alpha\in \K(X_\alpha)$ since $\K(X_\alpha)$ is a multiplicative basis, which further means $U_\alpha\cap V_\alpha\in\K(\mc{X})$. Lastly we note that $F\cap G$ is a finite subset of $\Lambda$, and we have $U\cap V\in \K(\mc{X})$.
\end{proof}

\begin{lemma}\label{union is sober} The space $\mc{X}$ is quasi-sober.
\end{lemma}

\begin{proof} This follows from observation (3) and the fact that each $X_\lambda$ is closed in $\mc{X}$.
\end{proof}

\begin{theorem}\label{disjoint union} The disjoint union $\mc{X}=\bigsqcup_{\lambda\in\Lambda}X_\lambda$ of a family of generalized spectral spaces with the disjoint union topology is also a generalized spectral space.
\end{theorem}

\begin{proof} Combine observation (5), Lemma \ref{basis}, and Lemma \ref{union is sober}.
\end{proof}

\begin{theorem}\label{finite coproduct in gspec} Let $\{X_j\}_{j=1}^n$ be a finite family of generalized spectral spaces. The coproduct of this family in $\mt{GSpec}$ is the disjoint union $\bigsqcup_{j=1}^nX_j$.
\end{theorem}

\begin{proof} Suppose we have a generalized spectral space $Z$ equipped with spectral maps $f_j:X_j\to Z$ for $j=1,\ldots, n$. Consider the diagram
\begin{center}
\begin{tikzcd}
                                                             &  & Z                                                                                    \\
                                                             &  &                                                                                      \\
X_j \arrow[rr, "\text{incl}_j"'] \arrow[rruu, "f_j"] &  & \bigsqcup_{j=1}^nX_j \arrow[uu, "{\exists !\, f\, ?}"', dotted]
\end{tikzcd}
\end{center}
For $x\in \bigsqcup_{j=1}^nX_j$, we must have $x\in X_j$ for some index $j$, and so we define $f(x):=f_j(x)$. This is the unique map making the above diagram commute. We show $f$ is a spectral map. 

Let $V$ be a compact open subset of $Z$. A computation gives $f^{-1}(V)=\bigsqcup_{j=1}^n f_j^{-1}(V)$. Since each $f_j$ is a spectral map, each preimage $f_j^{-1}(V)$ is compact open in $X_j$. Thus, applying observation (4) we see that $f^{-1}(V)$ is compact open in $\bigsqcup_{j=1}^nX_j$. So $f$ is a spectral map.
\end{proof}

\subsection{$\K\X$ obtains certain products}

Combining results from previous sections, we can now state the following.

\begin{lemma}\label{products} For a finite family $\{\A_j\}_{j=1}^n$ of cBCK-algebras we have lattice-isomorphisms
\[\K\X\Bigl(\bigcupdot_{j=1}^n\A_j\Bigr)\cong \K\Bigl(\bigsqcup_{j=1}^n\X(\A_j)\Bigr)\cong \prod_{j=1}^n \K\X(\A_j)\,.\]
\end{lemma}

\begin{proof} The first isomorphism follows from Theorem \ref{X(U) is coproduct}. The second follows from the fact that $\K$ sends coproducts in $\mt{GSpec}$ to products in $\mt{DL}_0$, together with Theorem \ref{finite coproduct in gspec}.
\end{proof}

\begin{corollary}\label{KX yields products} Let $\{\D_j\}_{j=1}^n$ be a finite family of distributive lattices such that, for each $j\in\{1,\ldots, n\}$, we have $\D_j\cong \K\X(\A_j)$ for some cBCK-algebra $\A_j$. Then the lattice $\P=\prod_{j=1}^n\D_j$ is in the image of $\K\X$.
\end{corollary}

\begin{proof} Let $\U=\bigcupdot_{j=1}^n\A_j$. Then by Lemma \ref{products} we have \[\K\X(\U)\cong \prod_{j=1}^n\K\X(\A_j)\cong \prod_{j=1}^n \D_j\cong \P\,.\]
\end{proof}

\begin{corollary}\label{specific products} Let $\{\D_j\}_{j=1}^n$ be a collection of distributive lattices such that each $\D_j$ is one of the five types from Theorem \ref{culmination}, and let $\P=\prod_{j=1}^n\D_j$. Then $\P$ is in the image of $\K\X$.
\end{corollary}

For example, pick $n\in\bb{N}$ and let $\D(n)$ be the divisor lattice of $n$. Since $\D(n)$ is a finite product of finite chains it is in the image of $\K\X$.

\end{document}